\documentclass[a4paper]{article}
\usepackage{amsthm}
\usepackage{amssymb}
 \usepackage{mathptmx}      
\usepackage{amsmath}
\usepackage{graphicx}
\usepackage{float}
\usepackage{amsfonts}
\usepackage[numbers,sort&compress]{natbib}
\usepackage[colorlinks,
            linkcolor=blue,
            anchorcolor=blue,
            citecolor=blue
            ]{hyperref}
            
\def\g{\boldsymbol{g}}

\def\p{\boldsymbol{p}}

\def\x{\boldsymbol{x}}
\def\y{\boldsymbol{y}}

\def\f{\boldsymbol{f}}

\def\g{\boldsymbol{g}}

\def\v{\boldsymbol{v}}
\def\h{\boldsymbol{h}}

\def\r{\boldsymbol{r}}
\def\u{\boldsymbol{u}}
\def\1{\boldsymbol{1}}

\newtheorem{definition}{Definition}
\newtheorem{proposition}{Proposition}
\newtheorem{remark}{Remark}
\newtheorem{assumption}{Assumption}
\newtheorem{theorem}{Theorem}
\newtheorem{corollary}{Corollary}
\newtheorem{lemma}{Lemma}
\newtheorem{condition}{Condition}

\begin{document}

\title{\Large $L^1$-optimality conditions for circular restricted three-body problems
}


\author{Zheng Chen\thanks{Laboratoire de Math\'ematiqeus d'Orsay, Univ. Paris-Sud, CNRS, Universit\'e Paris-Saclay, 91405 Orsay, France. E-mail:{zheng.chen@math.u-psud.fr}           
}      
}

\maketitle

\begin{abstract}
In this paper, the $L^1$-minimization for the translational motion of a spacecraft in a circular restricted three-body problem (CRTBP) is considered. Necessary conditions are derived by using the Pontryagin Maximum Principle, revealing the existence of bang-bang and singular controls. Singular extremals are detailed, recalling the existence of the Fuller phenomena according to the theories developed by Marchal in Ref.~\cite{Marchal:73} and Zelikin {\it et al.} in Refs.~\cite{Zelikin:94,Zelikin:03}.  The sufficient optimality conditions for the $L^1$-minimization problem with  fixed endpoints have been solved in Ref.~\cite{Caillau:15}.  In this paper, through constructing a parameterised family of extremals, some second-order sufficient conditions are established not only for the case that the final point is fixed but also for the case that  the final point lies on a smooth submanifold. In addition, the numerical implementation for the optimality conditions is presented. Finally, approximating  the Earth-Moon-Spacecraft system as a CRTBP, an $L^1$-minimization trajectory for the translational motion of a spacecraft is computed by employing a combination of a shooting method with a continuation method of Caillau {\it et al.} in Refs.~\cite{Caillau:12,Caillau:12time}, and the local optimality of the computed trajectory is tested thanks to the second-order optimality conditions established in this paper.
\end{abstract}

\section{ Introduction}
\label{intro}

As an increasing number of artificial satellites or spacecrafts have been and are being launched into deeper space since 1960s, the problem of controlling the translational motion of a spacecraft in the gravitational field of multiple celestial bodies such that some cost functionals are minimized or maximized arises in astronautics. The circular restricted three-body problem (CRTBP), which though as a degenerate model in celestial mechanics can capture the chaotic property  of $n$-body problem, is extensively used in the literature in recent years to study optimal trajectories in deeper space.  The controllability properties for the translational motion in CRTBPs are studied by Caillau {\it et al.} in Ref. \cite{Caillau:12time}, showing that there  exist  admissible controlled trajectories in an appropriate subregion of state space. The present paper is concerned with the $L^1$-minimization problem for the translational motion of a spacecraft in a CRTBP, which aims at minimizing the $L^1$-norm of control. Therefore, if the control is generated by propulsion systems which expel mass in a high speed to generate an opposite reaction force according to Newton's third law of motion, the $L^1$-minimization problem is referred to as the well-known fuel-optimal control problem in astronautics. The existence of the $L^1$-minimization solutions in CRTBPs can be obtained by a combination of Filippov theorem in Ref. \cite{Agrachev:04} and the technique in Ref. \cite{Gergaud:06} if we assume that admissible controlled trajectories remain in a fixed compact, see Ref.~\cite{Caillau:12}.

While in the planar case where the translational motion is restricted in a 2-dimensional (2D) plane,  the singular extremals and the corresponding chattering arcs are analyzed by Zelikin and Borisov in Ref.~\cite{Zelikin:03}, the synthesis of the solutions of singular extremals in 3-dimensional (3D) case, to the author's knowledge, is not covered up to the present time. Therefore, in this paper, in addition to an emphasis on the necessary conditions arising from the Pontryagin Maximum Principle (PMP), which reveals the existences of bang-bang and singular controls,  the solutions of singular extremals are investigated to show that the $L^1$-minimization trajectories in 3D case can exhibit Fuller or chattering phenomena according to the theories developed by Marchal in Ref.~\cite{Marchal:73} as well as by Zelikin and Borisov in Ref.~\cite{Zelikin:94}.

Even though one does not consider singular and chattering controls, the bang-bang type of control as well as the chaotic property in CRTBPs makes the computation of the $L^1$-minimization solutions a big challenge. To address this challenge, various numerical methods, {\it e.g.,} direct methods~\cite{Mingotti:09,Ross:07}, indirect methods~\cite{Caillau:12,Caillau:12time}, and hybrid methods~\cite{Ozimek:10}, have been developed recently. In this paper,  the indirect method, proposed by Caillau {\it et al.} in Refs. \cite{Caillau:12,Caillau:12time} to combine a  shooting method with a  continuation method, is employed  to compute the extremal trajectories of the $L^1$-minimization problem.  Based on this method, some kinds of fuel-optimal trajectories in a CRTBP are computed recently as well in Ref. \cite{Zhang:15}. Whereas, one can notice that the extremal trajectories computed by this indirect method cannot be guaranteed to be at least locally optimal unless  sufficient optimality conditions are satisfied. Thus, it is indeed crucial to test sufficient conditions to check if a computed trajectory realizes a local optimality, which is what is missing in the research of optimal trajectories in CRTBPs.

The sufficient  conditions for optimal control problems are widely studied in the literature in recent years, see Refs.~\cite{Agrachev:02,Poggiolini:04,Schattler:12,Noble:02,Caillau:15,Agrachev:04,Kupka:87,Sarychev:82,Sussmann:85,Bonnard:07} and the references therein.  Through defining an accessory finite dimensional problem in Refs.~\cite{Agrachev:02,Poggiolini:04}, some  sufficient  conditions are developed for optimal control problems with a polyhedral control set. In Ref.~\cite{Caillau:15},  two no-fold conditions are established for the $L^1$-minimization problem, which generalises  the results of Refs.~\cite{Schattler:12,Noble:02}. Assuming the endpoints are fixed, these two no-fold conditions are sufficient to guarantee a bang-bang extremal of the $L^1$-minimization problem to be a strong local optimizer (cf. Subsection \ref{Subse:sufficient1}). Whereas, in addition to the two no-fold conditions, a third condition has to be established once  the dimension of the constraint submanifold  of final states is not zero, see Refs.~\cite{Agrachev:02,Brusch:70,Wood:74}. In this paper, a parameterized family of extremals around a given extremal is constructed such that the third condition is managed to be related with Jacobi field  under some regularity assumptions (cf. Subsection \ref{Subse:sufficient2}). Then, it is shown that the propagation of Jacobi field is enough to test the sufficient optimality conditions (cf. Sect. \ref{SE:Procedure}).



The paper is organized as follows. In Sect. \ref{SE:Problem_Formulation}, the $L^1$-minimization problem is formulated in CRTBPs. Then, the necessary conditions are derived with an emphasis on singular solutions in Sect. \ref{SE:Necessary}. In Sect. \ref{SE:Sufficient}, a parameterized family of extremals is first constructed. Under some regularity assumptions, the sufficient conditions for the strong-local optimality of the nonsingular extremals with bang-bang controls are established. In Sect. \ref{SE:Procedure}, a numerical implementation for the optimality conditions is derived.  In Sect. \ref{SE:Numerical}, consider the Earth-Moon-Spacecraft system as a CRTBP, a transfer trajectory of a spacecraft from a circular geosynchronous orbit of the Earth  to a circular orbit around the Moon  is calculated to provide a bang-bang extremal, whose local optimality is tested thanks to the second-order optimality conditions developed in this paper.

\section{Definitions and notations}\label{SE:Problem_Formulation}


A CRTBP in celestial mechanics is generally defined as an  isolated dynamical system consisting of three gravitationally interacting bodies, $P_1$, $P_2$, and $P_3$, whose masses are denoted by $m_1$, $m_2$, and $m_3$, respectively, such that 1) the third mass $m_3$ is so much smaller than the other two that its gravitational influence on the motions of the other two is negligible and 2) the  two bodies, $P_1$ and $P_2$, move on their own circular orbits around their common centre of mass. 
Without loss of generality, we assume $m_1 > m_2$ and consider a rotating frame $OXYZ$ such that its origin is located at the barycentre of the two bodies $P_1$ and $P_2$, see Fig.~\ref{Fig:rotating_frame}. 
\begin{figure}[!ht]
 \includegraphics[trim=4.0cm 1.5cm 3.0cm 1.5cm, clip=true, width=4in]{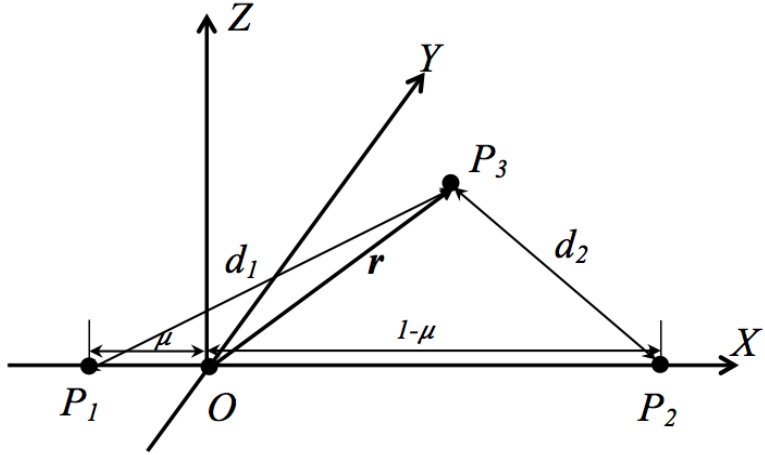}
 \caption[]{Rotating frame $OXYZ$ of the CRTBP.}
 \label{Fig:rotating_frame}
\end{figure}
The unit vector of $X$-axis is defined in such a way that it is collinear to the line between the two primaries $P_1$ and $P_2$ and points toward $P_2$, the unit vector of $Z$-axis is defined as the unit vector of the momentum vector of the motion of  $P_1$ and $P_2$, and the $Y$-axis is defined to complete a right-hand coordinate system.  It is advantageous to use non-dimensional parameters. Let $d_*$ be the distance between $P_1$ and $P_2$, and let $m_* = m_1 + m_2$, we denote by $d_*$ and $m_*$ the unit of length and mass, respectively. We also define the unit of time $t_*$ in such a way that the gravitational constant $G>0$ equals to one. Accordingly, one can obtain $$t_* = \sqrt{\frac{d_*^{3}}{Gm_*}}$$ through the usage of Kepler's third low. Then, denote by the superscript ``~$T$~" the transpose of matrices, if $\mu = m_2/m_*$, the two constant vectors $\r_1 = (-\mu,0,0)^T$ and $\r_2 = (1-\mu,0,0)^T$ denote the position of $P_1$ and $P_2$ in the rotating frame $OXYZ$, respectively.

\subsection{Dynamics}
{{In this paper, we denote the space of $n$-dimensional column vectors by $\mathbb{R}^n$ and the space of $n$-dimensional row vectors by $(\mathbb{R}^n)^*$.}} Let $t\in\mathbb{R}_+$ be the non-dimensional time and  let $\r \in \mathbb{R}^3$ and $\v  \in \mathbb{R}^3$ be the non-dimensional  position vector and velocity vector of $P_3$, respectively, in the rotating frame $OXYZ$.  Then, consider a spacecraft as the third mass point $P_3$ controlled by a finite-thrust propulsion system and let $m= m_3/m_*$, its state $\x\in\mathbb{R}^n$ ($n=7$) consists of position vector $\r$, velocity vector $\v$, and mass $m$, i.e., $\x = (\r,\v,m)$. Denote by the two constants $r_{m_1}>0$ and $r_{m_2}>0$ the radiuses of the two bodies $P_1$ and $P_2$, respectively, and denote by the constant $m_c>0$  the mass of the spacecraft without any fuel, we define the admissible subset for state $\x$ as
\begin{eqnarray}
\mathcal{X}=\big\{(\r,\v,m)\in\mathbb{R}^3 \times \mathbb{R}^3\times\mathbb{R}_+\ \arrowvert\   \|\r - \r_1\| > r_{m_1},\ \|\r-\r_2\| >r_{m_2},\ m \geq m_c\big\},\nonumber
\end{eqnarray}
where ``~$\|\cdot \|$~" denotes the Euclidean norm. Then, the differential equations for the controlled translational motion of the spacecraft in the CRTBP in the admissible set $\mathcal{X}$ for positive times can be written as
\begin{eqnarray}
\Sigma:
\begin{cases}
\dot{\r}(t) = \v(t),\\
\dot{\v}(t) = \h(\v(t)) + \g(\r(t))  + \frac{\boldsymbol{\tau}(t)}{m(t)},\\
\dot{m}(t) = -\beta{\parallel \boldsymbol{\tau}(t)\parallel},\label{EQ:mass}
\end{cases}
\label{EQ:Sigma}
\end{eqnarray}
with
\begin{eqnarray}
\h(\v) = \left[\begin{array}{ccc}
0 & 2 & 0\\
-2 & 0 & 0\\
0 & 0 & 0
\end{array}\right]\v, \ \boldsymbol{g}(\r) =  \left[\begin{array}{ccc}1& 0 & 0\\
0 & 1 & 0\\
0 & 0 & 0\end{array}\right] \r - \frac{1- \mu}{\|\r-\r_1\|^3}(\r - \r_1) - \frac{\mu}{\|\r -\r_2\|^3}(\r - \r_2),\nonumber
\end{eqnarray}
where   $\beta \geq 0$ is a scalar constant determined by the specific impulse of the engine equipped on the spacecraft and $\boldsymbol{\tau}\in\mathbb{R}^3$ is the thrust vector, taking values in 
\begin{eqnarray}
\mathcal{T} = \{\boldsymbol{\tau} \in\mathbb{R}^3\ \arrowvert \ \|\boldsymbol{\tau}\|\leq \tau_{max}\},\nonumber
\end{eqnarray}
where the constant $\tau_{max}>0$, in unit of ${m_*d_*}/{t_*^2}$, denotes the maximum magnitude of the thrust of the engine.


Denote by $\rho\in[0,1]$ the normalized mass flow rate of the engine, i.e., $\rho = \|\boldsymbol{\tau}\|/\tau_{max}$, and $\boldsymbol{\omega}\in\mathbb{S}^2$ the unit vector of the thrust direction, i.e., $\boldsymbol{\tau} = \rho\tau_{max}\boldsymbol{\omega}$, we then have that $\rho$ and $\boldsymbol{\omega}$ are control variables in the dynamics $\Sigma$ in Eq.~(\ref{EQ:Sigma}).  Let $\u=(\rho,\boldsymbol{\omega})$ and $\mathcal{U}=[0,1]\times\mathbb{S}^2$, we say $\mathcal{U}$ is the admissible set for the control $\u$.
Let us define the controlled vector field $\f$ on $\mathcal{X}\times\mathcal{U}$ by
\begin{eqnarray}
\f:\mathcal{X}\times\mathcal{U}\rightarrow \mathbb{R}^n,\ \f(\x,\rho,\boldsymbol{\omega}) = \f_0(\x) + \rho \f_1(\x,\boldsymbol{\omega}),\nonumber
\end{eqnarray}
where
\begin{eqnarray}
\f_0(\x) = \left(\begin{array}{c}
\v\\
\h(\v) + \g(\r)\\
{0}\end{array}\right),\ \ 
\f_1(\x,\boldsymbol{\omega}) = \left(\begin{array}{c}
\boldsymbol{0}\\
{\tau_{max}} \boldsymbol{\omega}/{m}\\
- {\tau_{max}}\beta  \end{array}\right).\nonumber
\end{eqnarray}
Then, the dynamics in Eq.~(\ref{EQ:Sigma}) can be rewritten as the control-affine form
\begin{eqnarray}
{\Sigma}:
\dot{\x}(t) =\f(\x(t),\rho(t),\boldsymbol{\omega}(t)) = \f_0(\x(t)) + \rho(t) \f_1(\x(t),\boldsymbol{\omega}(t)).
\label{EQ:system}
\end{eqnarray}

\subsection{$L^1$-minimization problem}

Given an $l\in\mathbb{N}$ such that $0< l \leq n$, we define the $l$-codimensional constraint submanifold on final state as
\begin{eqnarray}
\mathcal{M}=\{\x\in\mathcal{X}\ \arrowvert\ \phi(\x)=0\},
\label{EQ:final_manifold}
 \end{eqnarray}
where $\phi:\mathcal{X}\rightarrow \mathbb{R}^l$ denotes a twice continuously differentiable function of $\x$ and its expression depends on specific mission requirements, see an explicit example in Eq.~(\ref{EQ:function_phi}). Then, given a fixed initial state $\x_0\in\mathcal{X}$ and a fixed final time $t_f>0$, the $L^1$-minimization problem~\cite{Caillau:15} for the translational motion in the CRTBP consists of steering the system $\Sigma$ in $\mathcal{X}$ by a measurable control $(\rho(\cdot),\boldsymbol{\omega}(\cdot))\in\mathcal{U}$ on $[0,t_f]$  from the initial point $\x_0\in\mathcal{X}$ to a final point $\x_f\in\mathcal{M}$ such that the $L^1$-norm of control is minimized, i.e.,
\begin{eqnarray}
\int_0^{t_f} \rho(t) dt \rightarrow \text{min}.
 \label{EQ:cost_functional}
\end{eqnarray}
Note that the $L^1$-minimization problem is referred to as the fuel-minimum problem if $\beta > 0$. 

Controllability for the translational motion of the spacecraft in a CRTBP holds in an appropriate subregion of state space, see Ref.~\cite{Caillau:12}. Let $t_m > 0$ be the minimum time to steer the system $\Sigma$ by measurable controls $(\rho(\cdot),\boldsymbol{\omega}(\cdot))\in\mathcal{U}$ from the point $\x_0\in\mathcal{X}$ to a point $\x_f\in\mathcal{M}$. Then, assuming $t_f > t_m$ and that the admissible controlled trajectories of $\Sigma$ remain in a fixed compact, the existence of the $L^1$-minimization solutions can be obtained by  Filippov theorem \cite{Agrachev:04} since the convexity issues due to the $\rho$ term in the integrand of the cost in Eq.~(\ref{EQ:cost_functional}) can be dealt with as in Ref. \cite{Gergaud:06}. Therefore,  the PMP is applicable to formulate the following necessary conditions.



\section{Necessary conditions}\label{SE:Necessary}

\subsection{Pontryagin Maximum Principle}


According to the PMP in Ref. \cite{Pontryagin}, if a trajectory ${\x}(\cdot)\in\mathcal{X}$ associated with a measurable control ${\u}(\cdot)=(\rho(\cdot),\boldsymbol{\omega}(\cdot))$ in $\mathcal{U}$  on $[0,t_f]$ is an optimal one of the $L^1$-minimization problem, there exists a nonpositive real number $p^0$ and an absolutely continuous mapping ${t\mapsto\p(\cdot)\in T^*_{\x(\cdot)}\mathcal{X}}$ on $[0,t_f]$, satisfying $(\p,p^0) \neq 0$ and called adjoint state, such that almost everywhere on $[0,t_f]$   there holds
\begin{eqnarray}
\begin{cases}
\dot{\x}(t) = \frac{\partial H}{\partial \p}(\x(t),\p(t),p^0,\boldsymbol{u}(t)),\\
\dot{\p}(t) = -\frac{\partial H}{\partial \x}(\x(t),\p(t),p^0,\boldsymbol{u}(t)),
\end{cases}
\label{EQ:cannonical}
\end{eqnarray}
and
\begin{eqnarray}
H({\x}(t),{\p}(t),{p}^0,\boldsymbol{u}(t)) =\underset{\boldsymbol{\eta}(t)\in\mathcal{U}}{\text{max}} H({\x}(t),{\p}(t),{p}^0,\boldsymbol{\eta}(t)) ,
\label{EQ:maximum_condition}
\end{eqnarray}
where 
\begin{eqnarray}
H(\x,\p,p^0,\boldsymbol{u}) = \p\left[\f_0(\x) + \rho \f_1(\x,\boldsymbol{\omega})\right] + p^0 \rho,
\label{EQ:Hamiltonian}
\end{eqnarray}
 is the Hamiltonian. Moreover, the transversality condition asserts
\begin{eqnarray}
\boldsymbol{p}(t_f) =   \boldsymbol{\nu} d \phi(\x(t_f)),\label{EQ:Transversality_1}
\end{eqnarray}
where $\boldsymbol{\nu}\in(\mathbb{R}^l)^*$ is a constant vector whose elements are Lagrangian multipliers.

The 4-tuple $t\mapsto(\x(t),\p(t),p^0,\boldsymbol{u}(t))$  on $[0,t_f]$ is called an extremal.  Furthermore, an extremal is called a normal one if $p^0\neq 0$ and it is called an abnormal one if $p^0 = 0$. The abnormal extremals have been ruled out by Caillau {\it et al.} in Ref. \cite{Caillau:12}. Thus, in this paper only normal extremals are considered and $(\p,p^0)$ is normalized such that  $p^0 = -1$. According to the maximum condition in Eq.~(\ref{EQ:maximum_condition}), for every extremal $(\x(\cdot),\p(\cdot),p^0,\boldsymbol{u}(\cdot))$ on $[0,t_f]$, the corresponding extremal control  $\boldsymbol{u}(\cdot)$ is a function of $(\x(\cdot),\p(\cdot))$ on $[0,t_f]$, i.e., $\boldsymbol{u}(\cdot) = \boldsymbol{u}(\x(\cdot),\p(\cdot))$ on $[0,t_f]$. Thus, in the remainder of this paper, with some abuses of notations, we denote by $(\x(\cdot),\p(\cdot))\in T^*\mathcal{X}$ and $\boldsymbol{u}(\x(\cdot),\p(\cdot))\in\mathcal{U}$ on $[0,t_f]$ the normal extremal and the corresponding extremal control, respectively. And, we denote by ${H}(\x(\cdot),\p(\cdot))$ on $[0,t_f]$ the maximized Hamiltonian of the extremal $(\x(\cdot),\p(\cdot))$ on $[0,t_f]$, which is written as
\begin{eqnarray}
{H}(\x,\p) := {H}_0(\x,\p) + \rho(\x,\p) {H}_1(\x,\p),\nonumber
\end{eqnarray}
where ${H}_0 (\x,\p)  = \p\f_0(\x)$ and ${H}_1(\x,\p) = \p\f_1(\x,\boldsymbol{\omega}(\x,\p)) - 1$. 

Let us define by $\p_r\in T_{\r}\mathbb{R}^3$, $\p_v\in T_{\v}\mathbb{R}^3$, and $p_m\in T_{m}\mathbb{R}_+$ in such a way that $\p=(\p_r,\p_v,p_m)$, the maximum condition in Eq.~(\ref{EQ:maximum_condition}) implies 
\begin{eqnarray}
\boldsymbol{\omega}= \p_v / \parallel \p_v \parallel ,\  \text{if}\  \parallel \p_v\parallel\neq 0,
\label{EQ:Max_condition2}
\end{eqnarray}
and
\begin{eqnarray}
\begin{cases}\rho =1,\ \ \ \ \  \text{if}\ H_1 > 0,
\\
\rho = 0, \ \ \ \ \ \text{if}\ H_1 < 0.\\
\end{cases}
\label{EQ:Max_condition1}
\end{eqnarray}
Thus, the optimal direction of the thrust vector $\boldsymbol{\tau}$ is collinear to $\p_v$ that is well-known as the primer vector of Lawden~\cite{Lawden:63}. 
If the switching function $H_1$ has only isolated zeros along an extremal $(\x(\cdot),\p(\cdot))$ on $[0,t_f]$, this extremal is called a bang-bang one. 
\begin{definition}
Along a bang-bang extremal $(\x(\cdot),\p(\cdot))$ on $[0,t_f]$, an arc on a finite interval $[t_1,t_2]\subset [0,t_f]$ with $t_1 < t_2$ is called a maximum-thrust (or burn) arc if $\rho = 1$, otherwise it is called a zero-thrust (or coast) arc. 
\end{definition}

 \subsection{Singular solutions and chattering arcs}

An extremal $(\x(\cdot),\p(\cdot))$ on $[0,t_f]$ is said to be a singular one  if  $H_1(\x(\cdot),\p(\cdot)) \equiv 0$ for a finite interval $[t_1,t_2]\subseteq[0,t_f]$ with $t_1 < t_2$. Note that the maximum condition in Eq.~(\ref{EQ:maximum_condition}) is trivially satisfied for every $\rho \in[0,1]$ if $H_1 \equiv 0$. One can compute the optimal value of $\rho$ on singular arcs by repeatedly differentiating the identity $H_1 \equiv 0$ until $\rho$ explicitly appears. It is known from Ref.~\cite{Kelley:66} that  $\rho$ explicitly appears in the differentiation ${d^q H_1}/{d t^q}$ if and only if $q$ is an even integer, and the order of the singular arc is then designated as $q/2$.  
\begin{proposition}
Given a singular extremal $(\x(\cdot),\p(\cdot))$ on $[t_1,t_2]\subseteq[0,t_f]$ with $t_1 < t_2$, assume $\|\p_v (\cdot)\| \neq 0$ on $[t_1,t_2]$, we have that the order of the singular extremal is at least two.
\label{PR:singular_order}
\end{proposition}
\begin{proof}
Since $H_1 \equiv 0$ along a singular arc, differentiating $H_1$ with respect to time and using Poisson bracket, one obtains
\begin{eqnarray}
0 = H_{01}:=\big\{H_0,H_1\big\} = -\tau_{max}\frac{\p_v^T  [ \p_r + d{ \h(\v)} \p_v]}{m \parallel \p_v \parallel},
\label{EQ:H01}
\end{eqnarray}
where the notation ``~$\{\cdot,\cdot\}$~" denotes the Poisson bracket. Using Leibniz rule, Eq.~(\ref{EQ:H01}) implies
\begin{eqnarray}
H_{101} &:=& \big\{H_1,H_{01}\big\} = 0,\nonumber\\
H_{1001} &:=& \big\{H_1,\big\{H_0,H_{01}\big\}\big\}\nonumber\\
&=& \big\{-H_{01},H_{01}\big\} + \big\{H_0,H_{101}\big\} = 0.\nonumber
\end{eqnarray}
Then, the equality, $0 = H_{001} + \rho H_{101}$,
implies $H_{001} = 0$, whose implicit equation is
\begin{eqnarray}
H_{001} = \tau_{max} \frac{\p_v^T d\g(\r) \p_v + [\p_r + 2  d\h(\v)\p_v]^T[\p_r + d\h(\v)\p_v ] }{m \parallel \p_v \parallel}.\nonumber
\end{eqnarray}
 A direct calculation on this equation yields
\begin{eqnarray}
H_{0001} &:=& \big\{H_0,H_{001}\big\}\nonumber\\
 &=&\frac{\tau_{max}}{m\parallel \p_v \parallel} \Big\{  {{ \big[ \p_v^T d^2\g(\r) \p_v\big] }}\v - \p_v^T d\g(\r) [2\p_r + 3 d\h(\v) \p_v]\nonumber\\
 &-& [2  d\g(\r) \p_v + 3  d\h(\r)\p_r + 4(d\h(\v))^2\p_v  ]^T[\p_r +  d\h(\v)\p_v ]\Big\}.\nonumber
\end{eqnarray}
Eventually, one has $0 = \dot{H}_{0001} = H_{00001} + \rho H_{10001}$. Let
$\alpha_i$ ($i=1,2$) be defined by 
$$\cos(\alpha_i) = \frac{\p_v^T  (\r - \r_i)}{\parallel \p_v \parallel \parallel \r-\r_i \parallel},$$
 the explicit expression of $H_{10001}:=\{H_1,H_{0001}\}$, therefore, is
\begin{eqnarray}
H_{10001} &=&\tau_{max}  \frac{  \big[ \p_v^T d^2\g(\r) \p_v\big] \p_v}{m^2 \parallel \p_v \parallel^2}\nonumber\\
& =& 3\tau_{max} \frac{\parallel \p_v \parallel}{m^2} \left[\mu \cos \alpha_2\frac{3 - 5\cos^2 \alpha_2}{\parallel \r - \r_2 \parallel^4} + (1-\mu)\cos \alpha_1 \frac{3 - 5 \cos^2 \alpha_1}{\parallel \r - \r_1\parallel^4}\right].\nonumber
\end{eqnarray}
Note that the term $H_{10001}$ does not vanish identically on a singular extremal. Thus, the singular extremal is of order two according to Kelley's definition in Ref.~\cite{Kelley:66}, which proves the proposition. 
\end{proof}
This proposition for the 3D case expands the work in Ref.~\cite{Zelikin:03} where the motion of the spacecraft is restricted into a 2D plane and the work in Ref.~\cite{Robbins:65} where model of two-body problem ($\mu = 0$) is considered. 
 Note that Kelley's second-order necessary condition \cite{Kelley:66} in terms of  $\rho$ on singular arcs is $H_{10001} \leq 0$.
Let us define the singular submanifold $\mathcal{S}$ as
\begin{eqnarray}
\mathcal{S} = \big\{(\x,\p)\in T^*\mathcal{X}\ \arrowvert\ H_1=H_{01}=H_{001} = H_{0001} = 0,\  H_{10001} \leq 0\big\},\nonumber
\end{eqnarray}
we then obtain the following result.
\begin{corollary}[Fuller phenomenon, Zelikin and Borisov \cite{Zelikin:94}]
Let $\text{int}(\mathcal{S})$ be the interior of $\mathcal{S}$. Then, given every point $(\x,\p)\in\text{int}(\mathcal{S})$,  there exists a one parameter family of chattering solutions of Eqs.~(\ref{EQ:cannonical}--\ref{EQ:Hamiltonian}) passing through the point $(\x,\p)$ and another one parameter family of chattering solutions of Eqs.~(\ref{EQ:cannonical}--\ref{EQ:Hamiltonian}) coming out from the point $(\x,\p)$. 
\label{CO:fuller_phenomonon}
\end{corollary}
\noindent Though the efficient computation of chattering solutions is an open problem, see Ref.~\cite{Ghezzi:15}, Corollary \ref{CO:fuller_phenomonon} shows an insight into the control structure of the $L^1$-minimization trajectory, i.e., there exists a chattering arc when concatenating a singular arc with a nonsingular arc. The chattering arcs  may not be found by direct numerical methods when concatenating singular arcs with nonsingular arcs~\cite{Park:13}.

\section{Sufficient optimality conditions for bang-bang extremals}\label{SE:Sufficient}

Before studying the sufficient conditions for local optimality, we firstly give the definition of local optimality.
\begin{definition}[Local Optimality \cite{Poggiolini:04,Agrachev:02}] 
Given a fixed final time $t_f > 0$, an extremal trajectory $\bar{\x}(\cdot)\in \mathcal{X}$ associated with the extremal control $\bar{\boldsymbol{u}}(\cdot)=(\bar{\rho}(\cdot),\bar{\boldsymbol{\omega}}(\cdot))$ in $\mathcal{U}$ on $[0,t_f]$ is said to realize a weak-local optimality in $L^{\infty}$-topology (resp. a strong-local optimality in $C^0$-topology) if there exists an open neighborhood $\mathcal{W}_{\boldsymbol{u}}\subseteq \mathcal{U}$ of $\bar{\boldsymbol{u}}(\cdot)$ in $L^{\infty}$-topology (resp. an open neighborhood $\mathcal{W}_{\x}\subseteq \mathcal{X}$ of $\bar{\x}(\cdot)$ in $C^{0}$-topology) such that for every admissible controlled trajectory $\x(\cdot)\not\equiv\bar{\x}(\cdot)$ in $\mathcal{X}$ associated with the measurable control $\boldsymbol{u}(\cdot)=(\rho(\cdot),\boldsymbol{\omega}(\cdot))$ in $\mathcal{W}_{\boldsymbol{u}}$ on $[0,t_f]$ (resp.  for every admissible controlled trajectory $\x(\cdot)\not\equiv\bar{\x}(\cdot)$ in $\mathcal{W}_{\x}$ associated with the measurable control $\boldsymbol{u}(\cdot)=(\rho(\cdot),\boldsymbol{\omega}(\cdot)) $ in $\mathcal{U}$ on $[0,t_f]$)  with the boundary conditions $\x(0) = \bar{\x}(0)$ and $\x(t_f)\in\mathcal{M}$, there holds 
$$\int_{0}^{t_f}\rho(t) dt \geq \int_{0}^{t_f}\bar{\rho}(t) dt.$$
We say it realizes a strict weak-local (resp. strong-local) optimality if the strict inequality holds.
\label{DE:optimality}
\end{definition}
\noindent   Note that if a trajectory $\x(\cdot)\in\mathcal{X}$ on $[0,t_f]$ realizes a strong-local optimality, it automatically realizes a weak-local optimality. This section is concerned with establishing the sufficient conditions for the strong-local optimality.

\subsection{Parameterized family of extremals}

In this subsection, a family of extremals is constructed to be parameterized by $\p(0)\in T_{\x_0}^*\mathcal{X}$ such that the Poincar\'e-Cartan form $\p d\x - Hdt$ is exact on this family, which will be used to establish the sufficient optimality conditions later.

Let $\p_0 = \p(0)$, we define by
\begin{eqnarray}
\gamma:[0,t_f]\times T^*_{\x_0}\mathcal{X} \rightarrow T^*\mathcal{X},\ \gamma(t,\p_0) = (\x(t),\p(t)),\nonumber
\end{eqnarray}
the solution trajectory of Eqs.~(\ref{EQ:cannonical}--\ref{EQ:Hamiltonian}) such that $(\x_0,\p_0) = \gamma(0,\p_0)$. For every $\p_0\in T^*_{\x_0}\mathcal{X}$, we say $\gamma(\cdot,\p_0)$ on $[0,t_f]$ is an extremal. Note that at this moment we do not restrict any conditions on the final point of the extremal $\gamma(\cdot,\p_0)$ on $[0,t_f]$ for every $\p_0\in T^*\mathcal{X}$. 
\begin{definition}
We define $\bar{\p}_0\in T^*_{\x_0}\mathcal{X}$ in such a way that the extremal $\gamma(\cdot,\bar{\p}_0)$ at $t_f$ satisfies the final condition in Eq.~(\ref{EQ:final_manifold}) and transversality condition in Eq.~(\ref{EQ:Transversality_1}).
\end{definition}
\begin{definition}[Parameterized family of extremals]
Given the extremal $\gamma(\cdot,\bar{\p}_0)$ on $[0,t_f]$, let $\mathcal{P}\subset T_{\x_0}^*\mathcal{X}$ be an open neighbourhood of $\bar{\p}_0$, we say the subset
\begin{eqnarray}
\mathcal{F} = \big\{(\x(t),\p(t))\in T^*\mathcal{X}\ \arrowvert\ (\x(t),\p(t))=\gamma(t,\p_0),\ t\in[0,t_f],\ \p_0\in\mathcal{P}\big\},\nonumber
\end{eqnarray}
is a $\p_0$-parameterized family of extremals around the extremal $\gamma(\cdot,\bar{\p}_0)$ on $[0,t_f]$.
\end{definition}
\noindent Note that the open neighborhood $\mathcal{P}$ of $\bar{\p}_0$ in this paper can be shrunk whenever necessary.
Let
\begin{eqnarray}
\Pi: T^*\mathcal{X} \rightarrow \mathcal{X},\ \  (\x,\p) \mapsto \x,\nonumber
\end{eqnarray}
be the mapping that mapps a submanifold from the cotangent space $T^*\mathcal{X}$ onto  the state space $\mathcal{X}$, we say the mapping $\Pi$ is a canonical projection. 

An extremal ceases to be locally optimal if a focal point (or called a conjugate point if $l=n$ since in this case the endpoints are fixed) occurs~\cite{Bonnard:07}. According to Agrachev's approach in Ref.~\cite{Agrachev:04},   a focal point occurs on the extremal $\gamma(\cdot,\bar{\p}_0)$ at a time $t_c\in(0,t_f]$  if the projection of the family $\mathcal{F}$ loses its local diffeomorphism at $t_c$. We say the projection of the family $\mathcal{F}$ at  $t_c\in(0,t_f]$ is a fold singularity if it loses its local diffeomrophism at $t_c$. 
Thus, focal points are related to the fold singularities of the projection of the family $\mathcal{F}$.

\subsection{Sufficient conditions for the case of $l=n$}\label{Subse:sufficient1}

Given the extremal $(\bar{\x}(\cdot),\bar{\p}(\cdot)) = \gamma(\cdot,\bar{\p}_0)$ on $[0,t_f]$, without loss of generality, let the positive integer $k\in\mathbb{N}$ be the number of switching times $t_i$ ($i =1,2,\cdots,k$) such that $0 < t_1 < t_2 < \cdots < t_k < t_f$. 
\begin{assumption}
Along the extremal $(\bar{\x}(\cdot),\bar{\p}(\cdot)) = \gamma(\cdot,\bar{\p}_0)$ on $[0,t_f]$, each switching point at the switching time ${t}_i\in (0,t_f)$  is assumed to be a regular one, i.e., $H_1(\bar{\x}({t}_i),\bar{\p}({t}_i)) = 0$ and $H_{01}(\bar{\x}({t}_i),\bar{\p}({t}_i))\neq 0$ for $i=1,2,\cdots,k$.
\label{AS:Regular_Switching}
\end{assumption}
\noindent As a result of this assumption, if the subset $\mathcal{P}$ is small enough, the number of switching times on each extremal $\gamma(\cdot,\p_0)\in\mathcal{F}$ on $[0,t_f]$ keeps as $k$ and the $i$-th switching time of the extremals $\gamma(\cdot,\p_0)\in\mathcal{F}$ on $[0,t_f]$ is a smooth function of $\p_0$. Thus, we define by
\begin{eqnarray}
t_i: \mathcal{P}\rightarrow \mathbb{R}_+,\ \p_0 \mapsto t_i(\p_0),\nonumber
\end{eqnarray}
the $i$-th switching time of the  extremal $\gamma(\cdot,\p_0)\in\mathcal{F}$ on $[0,t_f]$.
Let 
\begin{eqnarray}
\mathcal{F}_i &=& \big\{(\x(t),\p(t))\in T^*\mathcal{X}\ \arrowvert\ \nonumber\\
&& (\x(t),\p(t)) = \gamma(t,\p_0),\ t \in (t_{i-1}(\p_0),t_i(\p_0)],\ \p_0 \in \mathcal{P}\big\}, \nonumber
\end{eqnarray}
for $\ i = 1,\ 2,\ \cdots,\ k,\ k+1$ with $t_0 = 0$ and $t_{k+1} = t_f$. If the subset $\mathcal{P}$ is small enough,  there holds
\begin{eqnarray}
\mathcal{F} = \mathcal{F}_1 \cup \mathcal{F}_2\cup \cdots \cup \mathcal{F}_k\cup \mathcal{F}_{k+1}.\nonumber
\end{eqnarray}
Let $(\x(\cdot,\p_0),\p(\cdot,\p_0))=\gamma(\cdot,\p_0)$ on $[0,t_f]$ be the extremals in $\mathcal{F}$. In order to avoid heavy notations, denote by $\delta(\cdot)$ the determinant of the matrix $\frac{\partial \x}{\partial \p_0}(\cdot,\bar{\p}_0)$ on $[0,t_f]$, i.e.,
$$\delta(\cdot) = \det\left[\frac{\partial \x}{\partial \p_0}(\cdot,\bar{\p}_0)\right],$$
on $[0,t_f]$. Note that the projection of the subset $\mathcal{F}_i$ at a time $t_c\in(t_i,t_{i+1})$ is a fold singularity  if $\delta(t_c) = 0$, as is shown by the typical picture for the occurrence of a conjugate point in Fig.~\ref{Fig:smooth_fold}. If $\delta(\cdot)\neq 0$ on $(t_i,t_{i+1})$, the projection of the subset $\mathcal{F}_i$ restricted to the domain $(t_i,t_{i+1})\times\mathcal{P}$  is a diffeomorphism, see Refs.~\cite{Schattler:12,Agrachev:04}. 
Let us define the following condition.
\begin{condition}
$\delta (\cdot) \neq 0$ on the open subintervals $({t}_i,{t}_{i+1})$ for $i=0,1,\cdots,k-1$ as well as on the semi-open subinterval $(t_{k},t_f]$.
\label{AS:Disconjugacy_bang}
\end{condition}
\begin{figure}[!ht]
 \includegraphics[trim=2.0cm 2.0cm 3.0cm 1.0cm, clip=true,  width=3.0in, angle=0]{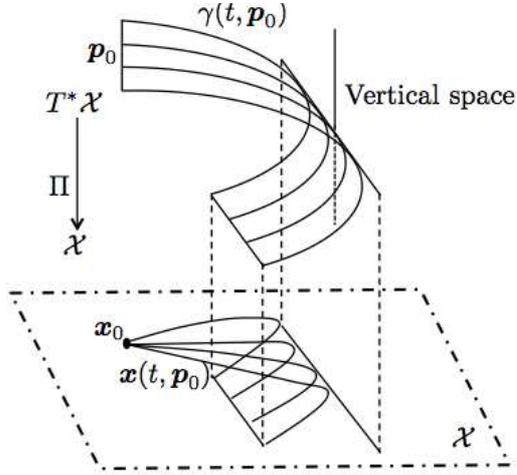}
\caption[]{A typical picture for a fold singularity of the projection of $\mathcal{F}$ onto the state space $\mathcal{X}$~\cite{Agrachev:04}.}
 \label{Fig:smooth_fold}
\end{figure}
\noindent Though this condition guarantees that both the restriction of $\Pi(\mathcal{F}_i)$ on $(t_{i-1},t_{i})\times\mathcal{P}$ for $i=1,2,\cdots,k$ and the restriction of $\Pi(\mathcal{F}_{k+1})$ on $(t_k,t_f]\times\mathcal{P}$  are local diffeomorphisms, it is not sufficient to guarantee that the projection of the family $\mathcal{F}$ restricted to the whole domain $(0,t_f]\times\mathcal{P}$ is a diffeomorphism as well, as Fig.~\ref{Fig:trans} shows that the flows $\x(t,\p_0)$ may intersect with each other near a switching time $t_i(\p_0)$. 
\begin{remark}
The behavior that the projection of $\mathcal{F}$ at a switching time $t_i$ is a fold singularity can be excluded by a transversal condition established by Noble and Sch\"attler in Ref.~\cite{Noble:02}. This transversal condition is reduced as $\delta(t_i-)\delta(t_i+)>0$ by Chen {\it et al.} in Ref.~\cite{Caillau:15}.
\end{remark}
\begin{figure}[!ht]
 \includegraphics[trim=2.0cm 1.5cm 1.0cm 1.0cm, clip=true, width=3.5in, angle=0]{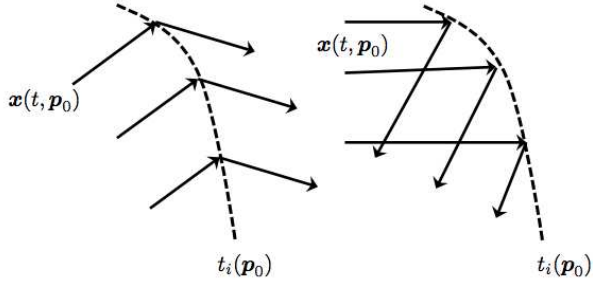}
\caption[]{The left plot denotes a diffeomorphism for the projection of $\mathcal{F}$ around the  switching time $t_i(\p_0)$, and the right plot denotes a fold singularity for the projection around the  switching time $t_i(\p_0)$ \cite{Noble:02,Schattler:12}.}
 \label{Fig:trans}
\end{figure}
\begin{condition}
$\delta (t_i-) \delta (t_{i}+) > 0$ for each switching time $t_i$ for $i=1,2,\cdots,k$.
\label{AS:Transversality}
\end{condition}
\noindent If this condition is satisfied, the projection  of the family $\mathcal{F}$ around each switching time $t_i(\p_0)$ is a diffeomorphism at least for a sufficiently small subset $\mathcal{P}$, see Ref.~\cite{Caillau:15}.
\begin{remark}
Given the extremal $(\bar{\x}(\cdot),\bar{\p}(\cdot)) = \gamma(\cdot,\bar{\p}_0)$ on $[0,t_f]$ such that every switching point is regular (cf. Assumption \ref{AS:Regular_Switching}) and  Conditions \ref{AS:Disconjugacy_bang} and \ref{AS:Transversality} are satisfied, if the subset $\mathcal{P}$ is small enough, every extremal $\gamma(\cdot,\p_0)$ on $[0,t_f]$ for $\p_0\in\mathcal{P}$ does not contain conjugate points. Then, for every $\p_0\in\mathcal{P}$,  we are able to 
construct a perturbed Lagrangian submanifold $\mathcal{L}_{\p_0}\subset T^*\mathcal{X}$ (cf. Theorem 21.3 in Ref.~\cite{Agrachev:04} or Appendix A in Ref.~\cite{Caillau:15}) around the extremal $\gamma(\cdot,\p_0)$ on $[0,t_f]$ such that 
\begin{description}
\item  $1)$ the projection  of the Lagrangian submanifold $\mathcal{L}_{\p_0}$ onto its image is a diffeomorphism; and
\item $2)$ the domain $\Pi(\mathcal{L}_{\p_0})$ is a tubular neighborhood of the extremal trajectory ${\x}(\cdot,\p_0)=\Pi(\gamma(\cdot,\p_0))$ on $[0,t_f]$.
\end{description}
\label{RE:lagrangian}
\end{remark}
\noindent As a result of this remark, one obtains the following remark.
\begin{remark}
If the subset $\mathcal{P}$ is small enough, let
\begin{eqnarray}
\mathcal{L} = \underset{\p_0\in\mathcal{P}}{\bigcap} \mathcal{L}_{\p_0},
\label{EQ:neighborhood}
\end{eqnarray}
it follows that
\begin{description}
\item $1)$  the projection of $\mathcal{L}$ onto its image is a diffeomorphism;
\item $2)$ the projection of $\mathcal{L}$ is a tubular neighborhood of the extremal trajectory $\Pi(\gamma(\cdot,\bar{\p}_0))$ on $[0,t_f]$; and
\item $3)$ there holds $\Pi(\mathcal{F})\subset \Pi(\mathcal{L})$ at every time $t\in[0,t_f]$.
\end{description}
\label{RE:neighborhood}
\end{remark}
\noindent Then, directly applying the theory of field of extremals (cf. Theorem 17.1 in Ref.~\cite{Agrachev:04}), one obtains the following result.
\begin{theorem}[Agrachev and Sachkov~\cite{Agrachev:04}]
Given the extremal $(\bar{\x}(\cdot),\bar{\p}(\cdot)) = \gamma(\cdot,\bar{\p}_0)$ on $[0,t_f]$ such that every switching point is regular (cf. Assumption \ref{AS:Regular_Switching}), let $(\rho(\cdot,\p_0),\boldsymbol{\omega}(\cdot,\p_0))\in\mathcal{U}$ be the optimal control function associated with the extremal $\gamma(\cdot,\p_0)\in\mathcal{F}$ on $[0,t_f]$. Then, if {\it Conditions \ref{AS:Disconjugacy_bang} and \ref{AS:Transversality}} are satisfied and if the subset $\mathcal{P}$ is small enough, every  extremal trajectory ${\x}(\cdot,\p_0)=\Pi(\gamma(\cdot,\p_0))$ on $[0,t_f]$ for $\p_0\in\mathcal{P}$ realizes a strict minimum cost among every admissible controlled trajectory $\x_*(\cdot)\in\Pi(\mathcal{L})$ associated with the measurable control $(\rho_*(\cdot),\boldsymbol{\omega}_*(\cdot))\in\mathcal{U}$ on $[0,t_f]$ with the same endpoints $\x(0,\p_0)=\x_*(0)$ and $\x(t_f,\p_0) = \x_*(t_f)$, i.e.,
\begin{eqnarray}
\int_0^{t_f}{\rho}(t,\p_0) dt \leq \int^{t_f}_0 \rho_*(t)dt,\nonumber
\end{eqnarray}
where the equality holds if and only if $\x_*(\cdot)\equiv \bar{\x}(\cdot)$ on $[0,t_f]$.
\label{CO:cor1}
\end{theorem}
\begin{proof}
According to Theorem $17.1$ in Ref.~\cite{Agrachev:04}, under the hypotheses of this theorem, every extremal trajectory $\x(\cdot,\p_0)$ on $[0,t_f]$ for $\p_0\in\mathcal{P}$ realizes a strict minimum cost among every admissible controlled trajectory $\x_*(\cdot)\in\Pi(\mathcal{L}_{\p_0})$ on $[0,t_f]$ with the same endpoints. Notice from Eq.~(\ref{EQ:neighborhood}) that  $\Pi(\mathcal{L})\subseteq\Pi(\mathcal{L}_{\p_0})$ at each time $t\in[0,t_f]$ for every $\p_0\in\mathcal{P}$, one proves this theorem. 
\end{proof}
\noindent Note that the endpoints of the $L^1$-minimization problem are fixed if $l=n$.
\begin{remark}
As a combination of Remark \ref{RE:neighborhood} and Theorem  \ref{CO:cor1}, one obtains that {\it Conditions} \ref{AS:Disconjugacy_bang} and \ref{AS:Transversality} are sufficient to guarantee the extremal trajectory $\bar{\x}(\cdot)$ on $[0,t_f]$ is a strict strong-local optimum (cf. Definition \ref{DE:optimality}) if $l=n$.  
\end{remark}
Under Assumption \ref{AS:Regular_Switching}, the projection of the family $\mathcal{F}$ near the switching time $t_i(\p_0)$ is a fold singularity if the strict inequality
$\delta (t_i-)\delta (t_{i}+) < 0$
is satisfied~\cite{Caillau:15}.
\begin{remark}
Given the extremal $\gamma(\cdot,\bar{\p}_0)$ on $[0,t_f]$ such that  each switching point is regular (cf. Assumption~\ref{AS:Regular_Switching}), conjugate points can occur not only on each smooth bang arc at a time $t_c\in (t_{i-1},t_i)$ if $\delta (t_c) = 0$ but also at each switching time $t_i$ if $\delta (t_i-) \delta (t_{i}+) < 0$.
\end{remark} 
\noindent The fact that conjugate points can occur at switching times generalizes the conjugate point theory developed by the classical variational methods for totally smooth extremals, see  Refs.~\cite{Bryson:69,Breakwell:65,Mermau:76,Wood:74}.

\subsection{Sufficient conditions for the case of $l<n$}\label{Subse:sufficient2}

In this subsection, we establish the sufficient optimality conditions for the case that the dimension of the final constraint submanifold $\mathcal{M}$ is not zero.

\begin{remark}
 If $l<n$,  to ensure the extremal trajectory $\bar{\x}(\cdot)$ on $[0,t_f]$ is a strict strong-local optimum, in addition to Conditions \ref{AS:Disconjugacy_bang} and \ref{AS:Transversality}, a further second-order condition (cf. Refs.~\cite{Wood:74,Brusch:70}) is required  to guarantee that every admissible controlled trajectory $\x_*(\cdot)\in\Pi(\mathcal{L})$ on $[0,t_f]$, not only with the same endpoints $\bar{\x}(0)=\x_*(0)$ and $\bar{\x}(t_f) = \x_*(t_f)$ but also with the boundary conditions $\bar{\x}(0)=\x_*(0)$ and $\x_*(t_f)\in\mathcal{M}\backslash\{\bar{\x}(t_f)\}$, has a bigger cost than the extremal trajectory $\bar{\x}(\cdot)$ on $[0,t_f]$. 
 \end{remark}
\noindent Let $\mathcal{N}\subset\mathcal{X}$ be the restriction of $\Pi(\mathcal{F})$ on $\{t_f\}\times\mathcal{P}$, i.e.,
$$\mathcal{N} = \big\{\x\in\mathcal{X}\ \arrowvert\ \x = \Pi(\gamma(t_f,\p_0)),\ \p_0\in\mathcal{P}\big\}.$$
Note that the mapping $\p_0\mapsto \x(t_f,\p_0)$ on the sufficiently small subset $\mathcal{P}$ is a diffeomorphism if $\delta (t_f)\neq 0$, which indicates that the subset $\mathcal{N}$ is an open neighborhood of $\bar{x}({t}_f)$ if Condition \ref{AS:Disconjugacy_bang} is satsfied. Thus, in the case of $l<n$, the subset $\mathcal{M}\cap \mathcal{N}\backslash\{\bar{\x}(t_f\}$ is not empty if $\delta(t_f)\neq 0$, see the sketch for a 2-dimensional state space in Fig~\ref{Fig:terminal_transversality}.
\begin{figure}[!ht]
\includegraphics[trim=1cm 1cm 1cm 1cm, clip=true,width=0.8\textwidth]{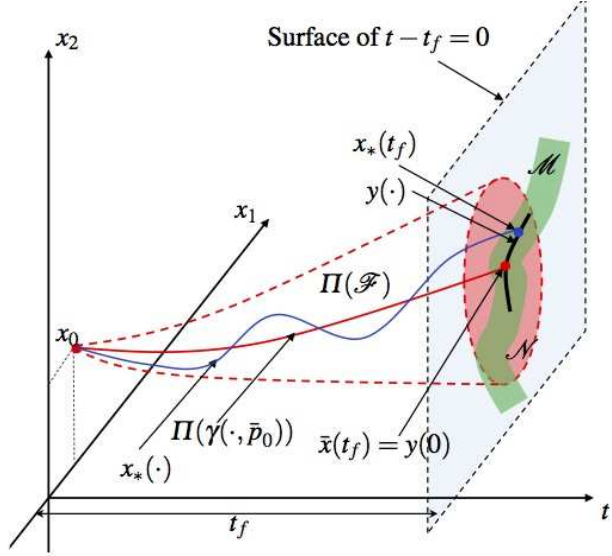}
 \caption[]{The relationship between $\mathcal{N}$ and $\mathcal{M}$.}
 \label{Fig:terminal_transversality}
\end{figure}
\let\emptyset\varnothing
For every sufficiently small subset $\mathcal{P}$, let us define by $\mathcal{Q}\subseteq \mathcal{P}$ a subset of all $\p_0\in\mathcal{P}$ satisfying $\Pi(\gamma(t_f,\p_0))\in\mathcal{M}\cap\mathcal{N}$, i.e.,
\begin{eqnarray}
\mathcal{Q} = \big\{\p_0\in\mathcal{P}\ \arrowvert\ \Pi(\gamma(t_f,\p_0))\in\mathcal{M}\cap\mathcal{N}\big\}.\nonumber
\end{eqnarray}
\noindent Note that for every $\p_0\in\mathcal{Q}$ there holds $\x_0 =\Pi(\gamma(0,\p_0))$ and $\Pi(\gamma(t_f,\p_0))\in\mathcal{M}$.
\begin{remark}
For every $\p_0\in\mathcal{Q}$, the extremal trajectory $\x(\cdot,\p_0) = \Pi(\gamma(\cdot,\p_0))$ on $[0,t_f]$  is an admissible controlled trajectory of the $L^1$-minimization problem.
\label{RE:admissible_control_trajectory}
\end{remark}
\begin{definition}
Given the extremal $(\bar{\x}(\cdot),\bar{\p}(\cdot)) = \gamma(\cdot,\bar{\p}_0)$ on $[0,t_f]$ and  a small $\varepsilon > 0$,  let $l< n$. Then, we define by $\y: [-\varepsilon,\varepsilon]\rightarrow \mathcal{M}\cap\mathcal{N},\ \eta\mapsto \y(\eta)$ a twice continuously differentiable curve on $\mathcal{M}\cap\mathcal{N}$ such that $\y(0)=\bar{\x}(t_f)$.
\label{DE:smooth_curve}
\end{definition}
\begin{lemma}
Given the extremal $\gamma(\cdot,\bar{\p}_0)$ on $[0,t_f]$ such that each switching point  is regular (cf. Assumption \ref{AS:Regular_Switching}) and {\it Conditions} \ref{AS:Disconjugacy_bang} and \ref{AS:Transversality} are satisfied, let $l<n$. Then, if the subset $\mathcal{P}$ is small enough, for every smooth curve $\y(\cdot)\in\mathcal{M}\cap\mathcal{N}$ on $[-\varepsilon,\varepsilon]$, there exists a smooth path $\eta\mapsto \p_0(\eta)$ on $[-\varepsilon,\varepsilon]$ in $\mathcal{Q}$  such that $\y(\cdot) = \Pi(\gamma(t_f,\p_0(\cdot)))$ on $[-\varepsilon,\varepsilon]$. 
\label{LE:smooth_path}
\end{lemma}
\begin{proof}
Note that the mapping $\p_0\mapsto \x(t_f,\p_0)$ restricted to the subset $\mathcal{Q}$ is a diffeomorphism under the hypotheses of the lemma. Then, according to the {\it inverse function theorem}, the lemma is proved. 
\end{proof}

\begin{definition}
Define a path ${\boldsymbol{\lambda}}:[-\varepsilon,\varepsilon]\rightarrow T^*_{\y(\cdot)}\mathcal{X},\ \eta\mapsto \boldsymbol{\lambda}(\eta)$ in such a way that $(\y(\cdot),{\boldsymbol{\lambda}}(\cdot)) = \gamma(t_f,\p_0(\cdot))$ on $[-\varepsilon,\varepsilon]$. Then, for every $\xi\in[-\varepsilon,\varepsilon]$, we define by $J:[-\varepsilon,\varepsilon]\rightarrow \mathbb{R},\ \xi\mapsto J(\xi)$  the  integrand of the Poincar\'e-Cartan form $\p d\x - H dt$ along the extremal lift $(\y(\cdot),\boldsymbol{\lambda}(\cdot))$ on $[0,\xi]$, i.e., 
\begin{eqnarray}
J(\xi) = \int_0^{\xi}\boldsymbol{\lambda}(\eta)\y^{\prime}(\eta) - H(\y(\eta),\boldsymbol{\lambda}(\eta))\frac{d t_f}{d\eta}d\eta,\ \xi\in[-\varepsilon,\varepsilon].
\label{EQ:J_xi}
\end{eqnarray}
\end{definition}
\begin{proposition}
In the case of $l<n$, given the extremal $(\bar{\x}(\cdot),\bar{\p}(\cdot)) = \gamma(\cdot,\bar{\p}_0)$ on $[0,t_f]$ such that each switching point  is regular (cf. Assumption \ref{AS:Regular_Switching}) and {\it Conditions} \ref{AS:Disconjugacy_bang} and \ref{AS:Transversality} are satisfied, assume $\varepsilon > 0$ is small enough. Then, the extremal trajectory $\bar{\x}(\cdot)$ on $[0,t_f]$ is a strict strong-local optimality (cf. Definition \ref{DE:optimality}) if and only if there holds 
\begin{eqnarray}
J(\xi) > J(0),\ \xi\in[-\varepsilon,\varepsilon]\backslash\{0\},
\label{EQ:lemma1_compare}
\end{eqnarray}
for every smooth curve $\y(\cdot)\in\mathcal{M}\cap\mathcal{N}$ on $[-\varepsilon,\varepsilon]$.
\label{CO:J_xi_J_0}
\end{proposition}
\begin{proof}
Let us first prove that, under the hypotheses of this proposition, Eq.~(\ref{EQ:lemma1_compare}) is a sufficient condition for the strict strong-local optimality of the extremal trajectory $\bar{\x}(\cdot)$ on $[0,t_f]$.  Denote by $\x_*(\cdot)$ in $\Pi(\mathcal{L})$ on $[0,t_f]$ be an admissible controlled trajectory with the boundary conditions $\x_*(0) = \bar{\x}(0)$ and $\x_*(t_f)\in\mathcal{M}\cap\mathcal{N}\backslash\{\bar{\x}(t_f)\}$. Let $(\rho_*(\cdot),\boldsymbol{\omega}_*(\cdot))\in\mathcal{U}$ and $(\rho(\cdot,\p_0),\boldsymbol{\omega}(\cdot,\p_0))\in\mathcal{U}$ on $[0,t_f]$ be the measurable control and the optimal control associated with $\x_*(\cdot)$ and $\x(\cdot,\p_0)$ on $[0,t_f]$, respectively. According to Definition \ref{DE:smooth_curve} and Lemma \ref{LE:smooth_path}, for every final point $\x_*(t_f)\in\mathcal{M}\cap\mathcal{N}\backslash\{\bar{\x}(t_f)\}$, there must exist a $\xi\in[-\varepsilon,\varepsilon]\backslash\{0\}$ and a smooth path $\p_0(\cdot)\in\mathcal{Q}$ associated with the smooth curve $\y(\cdot)\in\mathcal{M}\cap\mathcal{N}$ on $[-\varepsilon,\varepsilon]$ such that  $\y(0) = \bar{\x}(t_f) = \Pi(\gamma(t_f,\p_0(\xi)))$ and $\y(\xi) = \x_*(t_f)=\Pi(\gamma(t_f,\p_0(\xi)))$.  Since the trajectory $\x_*(\cdot)$ on $[0,t_f]$ has the same endpoints with the extremal trajectory $\x(\cdot,\p_0(\xi))=\Pi(\gamma(\cdot,\p_0(\xi)))$ on $[0,t_f]$, according to Theorem \ref{CO:cor1}, one obtains
\begin{eqnarray}
\int_0^{t_f} \rho_*(t)dt \geq \int_0^{t_f}\rho(t,\p_0(\xi))dt,
\label{EQ:rho_*_rho_xi}
\end{eqnarray}
where the equality holds if and only if $\x_*(\cdot)\equiv {\x}(\cdot,\p_0(\xi))$ on $[0,t_f]$.

Note that the four paths $(\x_0,\p_0(\cdot))$ on $[0,\xi]$, $\gamma(\cdot,\bar{\p}_0)$ on $[0,t_f]$, $(\x(\cdot,\p_0(\xi)),\p(\cdot,\p_0(\xi))) = \gamma(\cdot,\p_0(\xi))$ on $[0,t_f]$, and $(\y(\cdot),\boldsymbol{\lambda}(\cdot))$ on $[0,\xi]$ constitute a closed curve on the family $\mathcal{F}$. Since the integrand of the Poincar\'e-Cartan form $\p d\x - H dt$ is closed on $\mathcal{F}$, see Refs.~\cite{Agrachev:04,Schattler:12,Caillau:15}, one obtains
\begin{eqnarray}
&&J(\xi) + \int_0^{t_f}\big[\bar{\p}(t)\dot{\bar{\x}}(t) - H(\bar{\x}(t),\bar{\p}(t))\big]dt\nonumber\\
&=& \int_0^{t_f}\big[\p(t,\p_0(\xi))\dot{\x}(t,\p_0(\xi)) - H(\x(t,\p_0(\xi)),\p(t,\p_0(\xi)))\big]dt\nonumber\\
&+& \int_0^{\xi} \Big[\p_0(\eta)\frac{d {\x}_0}{d\eta} - H(\x_0,\p_0(\eta))\frac{d t_0}{d\eta}\Big]d\eta,
\label{EQ:compare1111}
\end{eqnarray}
where $t_0 = 0$. Since $\x_0$ is fixed, one obtains $$ \int_0^{\xi} \Big[\p_0(\eta)\frac{d {\x}_0}{d\eta} - H(\x_0,\p_0(\eta))\frac{d t_0}{d\eta}\Big]d\eta = 0$$ for every $\xi\in[-\varepsilon,\varepsilon]$.
Then, taking into account Eq.~(\ref{EQ:Hamiltonian}), a combination of Eq.~(\ref{EQ:compare1111}) with Eq.~(\ref{EQ:rho_*_rho_xi}) leads to
\begin{eqnarray}
\int_0^{t_f} \bar{\rho}(t) dt &=& \int_0^{t_f}\big[\bar{\p}(t)\dot{\bar{\x}}(t) - H(\bar{\x}(t),\bar{\p}(t))\big]dt\nonumber\\
&=& - J(\xi)  + \int_0^{t_f}\big[\p(t,\p_0(\xi))\dot{\x}(t,\p_0(\xi)) - H(\x(t,\p_0(\xi)),\p(t,\p_0(\xi)))\big]dt\nonumber\\
&=& - J(\xi) + \int_0^{t_f}\rho(t,\p_0(\xi))dt\nonumber\\
&\leq & - J(\xi) + \int_0^{t_f}\rho_*(t)dt.
\label{EQ:lemma1_compare_new}
\end{eqnarray}
Since $J(0) = 0$, Eq.~(\ref{EQ:lemma1_compare}) implies the strict inequality
\begin{eqnarray}
\int_0^{t_f} \bar{\rho}(t) dt <  \int_0^{t_f}\rho_*(t)dt,
\label{EQ:compare111}
\end{eqnarray}
holds if $\xi\neq 0$ or $\x_*(t_f)\neq \bar{\x}(t_f)$. For the case of $\x_*(t_f)=\bar{\x}(t_f)$, Eq.~(\ref{EQ:compare111}) is satisfied as well according to Theorem \ref{CO:cor1}, which proves that Eq.~(\ref{EQ:lemma1_compare}) is a sufficient condition.

Next, let us prove that Eq.~(\ref{EQ:lemma1_compare}) is a necessary condition. Assume Eq.~(\ref{EQ:lemma1_compare}) is not satisfied, i.e., there exists a smooth curve $\y(\cdot)\in\mathcal{M}\cap\mathcal{N}$ on $[-\varepsilon,\varepsilon]$ and a $\xi\in[-\varepsilon,\varepsilon]\backslash\{0\}$ such that $J(\xi) \leq J(0) = 0$. Then, according to Eq.~(\ref{EQ:lemma1_compare_new}), one obtains 
$$\int_0^{t_f}\bar{\rho}(t)dt \geq \int_0^{t_f}\rho(t,\p_0(\xi)) dt.$$
Note that the extremal trajectory $\Pi(\gamma(\cdot,\p_0(\xi)))$ in $\Pi(\mathcal{F})\subset\Pi(\mathcal{L})$ is an admissible controlled trajectory of the $L^1$-minimization problem (cf. Remark \ref{RE:admissible_control_trajectory}).  Thus, the proposition is proved.
\end{proof}
\begin{proposition}
Given the extremal $(\bar{\x}(\cdot),\bar{\p}(\cdot))=\gamma(\cdot,\bar{\p}_0)$ on $[0,t_f]$ such that each switching point is regular (cf. Assumption \ref{AS:Regular_Switching}) and {\it Conditions}  \ref{AS:Disconjugacy_bang} and \ref{AS:Transversality} are satisfied, let $l<n$. Then, if $\varepsilon>0$ is small enough,  the inequality $J^{\prime\prime}(0) \geq 0$ (resp. the strict inequality $J^{\prime\prime}(0) > 0$) for every smooth curve $\y(\cdot)\in\mathcal{M}\cap\mathcal{N}$ on $[-\varepsilon,\varepsilon]$ is a necessary condition (resp. a sufficient condition) for the strict strong-local optimality of the extremal trajectory $\bar{\x}(\cdot)$ on $[0,t_f]$.
 \label{PR:proposition_J2}
\end{proposition}
\begin{proof}
Since the final time $t_f$ is fixed, Eq.~(\ref{EQ:J_xi}) is reduced as
\begin{eqnarray}
J(\xi) = \int_{0}^{\xi}\boldsymbol{\lambda}(\eta){\y^{\prime}}(\eta)d\eta.\nonumber
\end{eqnarray}
Taking derivative of $J(\xi)$ with respect to $\xi$ yields
\begin{eqnarray}
J^{\prime}(\xi) = \boldsymbol{\lambda}(\xi)\cdot \y^{\prime}(\xi).
\label{EQ:J_prime}
\end{eqnarray}
Note that ${\boldsymbol{\lambda}}(0) = \bar{\p}(t_f)$. Taking into account Eq.~(\ref{EQ:Transversality_1}), for every smooth curve $\y(\cdot)\in\mathcal{M}\cap\mathcal{N}$ on $[-\varepsilon,\varepsilon]$, we have ${ J^{\prime}}(0) = {\boldsymbol{\lambda}}(0) \y^{\prime}(0) = 0$ since  ${\y^{\prime}}(0)$ is a tangent vector of the submanifold $\mathcal{M}$ at $\bar{\x}(t_f)$. Then, according to Proposition \ref{CO:J_xi_J_0}, this proposition is proved. 
\end{proof}

\begin{definition}
Given the extremal $(\bar{\x}(\cdot),\bar{\p}(\cdot)) = \gamma(\cdot,\bar{\p}_0)$ on $[0,t_f]$, denote by $\bar{\boldsymbol{\nu}}\in(\mathbb{R}^l)^*$ the vector of the Lagrangian multipliers of this extremal such that
$$\bar{\p}(t_f) = \bar{\boldsymbol{\nu}}{d\phi(\bar{\x}(t_f))}.$$
\label{DE:Lagrangian_Multiplier}
\end{definition}

\begin{proposition}
In the case of $l< n$, given the extremal $(\bar{\x}(\cdot),\bar{\p}(\cdot)) =\gamma(\cdot,\bar{\p}_0)$ on $[0,t_f]$ such that each switching point is regular (cf. Assumption \ref{AS:Regular_Switching}), assume {\it Conditions} \ref{AS:Disconjugacy_bang} and \ref{AS:Transversality} are satisfied. Then,  the inequality ${ J^{\prime\prime}}(0) \geq 0$ (resp. strict inequality ${ J^{\prime\prime}}(0) > 0$) is satisfied for every smooth curve $\y(\cdot)\in\mathcal{M}\cap\mathcal{N}$ on $[-\varepsilon,\varepsilon]$ if and only if there holds
 \begin{eqnarray}
\boldsymbol{\zeta}^T\left\{\frac{\partial \p^T(t_f,\bar{\p}_0)}{\partial \p_0} \left[ \frac{\partial \x(t_f,\bar{\p}_0)}{\partial \p_0}\right]^{-1} - \bar{\boldsymbol{\nu}}d^2\phi(\bar{\x}(t_f))\right\}\boldsymbol{\zeta} \geq 0\ \text{(resp.}\ >0\text{)},\nonumber
\end{eqnarray}
for every tangent vector $\boldsymbol{\zeta}\in T_{\bar{\x}(t_f)}\mathcal{M}\backslash\{0\}$.
\label{LE:lemma2}
\end{proposition}
\begin{proof}
Differentiating $J^{\prime}(\xi)$ in Eq.~(\ref{EQ:J_prime}) with respect to $\xi$ yields
\begin{eqnarray}
{J^{\prime\prime}}(\xi) &=& {\boldsymbol{\lambda}^{\prime}(\xi)}{\y^{\prime}(\xi)} + \boldsymbol{\lambda}(\xi){\y^{\prime\prime}(\xi)}.
\label{EQ:d2Jdxi2}
\end{eqnarray}
Then, differentiating $\phi(\y(\xi))$ with respect to $\xi$ yields
\begin{eqnarray}
\frac{d  }{d\xi}\phi(\y(\xi))  &=&{{ d\phi(\y(\xi))}} {\y^{\prime}(\xi)}= 0,\label{EQ:dphidxi}\nonumber\\
\frac{d^2  }{d\xi^2}\phi(\y(\xi))  &=& [d^2\phi(\y(\xi))\y^{\prime}(\xi)]\y^{\prime}(\xi)
 + {d \phi(\y(\xi))}{\y^{\prime\prime}(\xi)}= 0.
\label{EQ:dphi2dxi2}
\end{eqnarray}
Since $(\bar{\x}(t_f),\bar{\p}(t_f)) = (\y(0),\boldsymbol{\lambda}(0))$, according to the definition of the vector $\bar{\boldsymbol{\nu}}$ in Definition \ref{DE:Lagrangian_Multiplier}, one immediately has $\boldsymbol{\lambda}(0) = \bar{\boldsymbol{\nu}}{d\phi(\y(0))}$.
Thus,  multiplying $\bar{\boldsymbol{\nu}}$ on both sides of Eq.~(\ref{EQ:dphi2dxi2}) and fixing $\xi = 0$, we obtain
\begin{eqnarray}
\bar{\boldsymbol{\nu}}\frac{d^2  \phi(\y(0))}{d\xi^2}  &=& \boldsymbol{\lambda}(0){\y^{\prime\prime}(0)}  + \bar{\boldsymbol{\nu}}
\left[d^2\phi(\y(0))\y^{\prime}(0)\right] \y^{\prime}(0)
\nonumber\\
&=&  {\boldsymbol{\lambda}(0){\y^{\prime\prime}(0)}  +
\left[\y^{\prime}(0)\right]^{T}\left[ \bar{\boldsymbol{\nu}}d^2\phi(\y(0))\right]\y^{\prime}(0) }
\nonumber\\
&=& 0.\nonumber
\end{eqnarray}
Substituting this equation into Eq.~(\ref{EQ:d2Jdxi2}) yields
\begin{eqnarray}
{ J^{\prime\prime}}(0) = \boldsymbol{\lambda}^{\prime}(0){\y^{\prime}(0)}  -\left[\y^{\prime}(0)\right]^{T}\left[\bar{\boldsymbol{\nu}}d^2\phi(\y(0))\right]\y^{\prime}(0).
\label{EQ:d2Jdxi20}
\end{eqnarray}
Note that we have
\begin{eqnarray}
{\y^{\prime}(\xi)} &=& \frac{d\x(t_f,\p_0(\xi))}{d\xi} = \frac{\partial \x(t_f,\p_0(\xi))}{\partial \p_0}\left[{\p_0^{\prime}(\xi)}\right]^T,\nonumber\\
\left[{\boldsymbol{\lambda}^{\prime}(\xi)} \right]^T&=&\frac{d\p^T(t_f,\p_0(\xi))}{d\xi}= \frac{\partial \p^T(t_f,\p_0(\xi))}{\partial \p_0} \left[{\p_0^{\prime}(\xi)}\right]^T.
\label{EQ:dlambdadxi}
\end{eqnarray}
 Since the matrix $ \frac{\partial \x(t_f,\p_0(\xi))}{\partial \p_0}$  is nonsingular if {\it Condition} \ref{AS:Disconjugacy_bang} is satisfied, we have $$\left[{\p_0^{\prime}(\xi)} \right]^T= \left[ \frac{\partial \x(t_f,\p_0(\xi))}{\partial \p_0}\right]^{-1}{\y^{\prime}(\xi)}.$$
Substituting this equation into Eq.~(\ref{EQ:dlambdadxi})
yields
$$\left[{\boldsymbol{\lambda}^{\prime}(\xi)}\right]^T = \frac{\partial \p^T(t_f,\p_0(\xi))}{\partial \p_0} \left[ \frac{\partial \x(t_f,\p_0(\xi))}{\partial \p_0}\right]^{-1}{\y^{\prime}(\xi)}.$$
Again, substituting this equation into Eq.~(\ref{EQ:d2Jdxi20}) and taking into account $\bar{\p}_0=\p_0(0)$ and $\bar{\x}(t_f) = \y(0)$, we eventually get that for every smooth curve $\y(\cdot)\in\mathcal{M}\cap\mathcal{N}$ on $[-\varepsilon,\varepsilon]$ there holds
\begin{eqnarray}
J^{\prime\prime} (0) = \left[\y^{\prime}(0)\right]^T\Big\{\frac{\partial \p^T(t_f,\bar{\p}_0)}{\partial \p_0} \left[ \frac{\partial \x(t_f,\bar{\p}_0)}{\partial \p_0}\right]^{-1} - \bar{\boldsymbol{\nu}}d^2\phi(\bar{\x}(t_f)) \Big\}\y^{\prime}(0).\nonumber
\end{eqnarray}
Note that the vector $\y^{\prime}(0)$ can be an arbitrary vector in the tangent space $T_{\bar{\x}(t_f)}\mathcal{X}\backslash\{0\}$, one proves this proposition. 
\end{proof}

\begin{condition}
Given the extremal $(\bar{\x}(\cdot),\bar{\p}(\cdot)) = \gamma(\cdot,\bar{\p}_0)$ on $[0,t_f]$,  let 
\begin{eqnarray}
\boldsymbol{\zeta}^T\left\{\frac{\partial \p^T(t_f,\bar{\p}_0)}{\partial \p_0} \left[ \frac{\partial \x(t_f,\bar{\p}_0)}{\partial \p_0}\right]^{-1} - \bar{\boldsymbol{\nu}}d^2\phi(\bar{\x}(t_f))\right\}\boldsymbol{\zeta} > 0,\nonumber
\end{eqnarray}
be satisfied for every vector $\boldsymbol{\zeta}\in T_{\bar{\x}(t_f)}\mathcal{M}\backslash\{0\}$.
\label{AS:terminal_condition}
\end{condition}
\noindent Then, as a combination {\it Propositions} \ref{PR:proposition_J2} and \ref{LE:lemma2}, we eventually obtain the following result.
\begin{theorem}
Given the extremal $(\bar{\x}(\cdot),\bar{\p}(\cdot)) = \gamma(\cdot,\bar{\p}_0)$ on $[0,t_f]$ such that every switching point is regular (cf. Assumption \ref{AS:Regular_Switching}), let $l<n$. Then, if {\it Conditions} \ref{AS:Disconjugacy_bang}, \ref{AS:Transversality}, and \ref{AS:terminal_condition} are satisfied, the extremal trajectory $\bar{\x}(\cdot)$ on $[0,t_f]$ realizes a strict strong-local optimality (cf. Definition \ref{DE:optimality}).
\label{TH:optimality}
\end{theorem}
Consequently, in the case of $l<n$, {\it Conditions} \ref{AS:Disconjugacy_bang}, \ref{AS:Transversality}, and \ref{AS:terminal_condition} are sufficient to guarantee a bang-bang extremal with regular switching points to be a strict strong-local optimum. In next section, the numerical implementation for these three conditions will be derived.

\section{Numerical implementation for sufficient optimality conditions}\label{SE:Procedure}

Once the extremal $(\bar{\x}(\cdot),\bar{\p}(\cdot))=\gamma(\cdot,\bar{\p}_0)$ on $[0,t_f]$ is computed,  according to Definition \ref{DE:Lagrangian_Multiplier},  the vector $\bar{\boldsymbol{\nu}}$ of Lagrangian multipliers in Condition \ref{AS:terminal_condition} can be computed by
\begin{eqnarray}
\bar{\boldsymbol{\nu}} = \bar{\p}(t_f){d\phi^T(\bar{\x}(t_f))} \left[{d\phi(\bar{\x}(t_f))} {d\phi^T(\bar{\x}(t_f))}\right]^{-1}.\label{EQ:numerical_nu}
\end{eqnarray}
\begin{definition}
We define by $\boldsymbol{C}\in\mathbb{R}^{n\times (n-l)}$ a full-rank matrix such that its columns constitute a basis of the tangent space $T_{\bar{\x}(t_f)}\mathcal{M}$.\label{DE:definition_C}
\end{definition}
\noindent Then, one immediately gets that Condition \ref{AS:terminal_condition} is satisfied if and only if there holds
\begin{eqnarray}
\boldsymbol{C}^T\left\{\frac{\partial \p^T(t_f,\bar{\p}_0)}{\partial \p_0} \left[ \frac{\partial \x(t_f,\bar{\p}_0)}{\partial \p_0}\right]^{-1} - \bar{\boldsymbol{\nu}}d^2\phi(\bar{\x}(t_f))\right\}\boldsymbol{C} \succ 0.
\label{EQ:positive_definite}
\end{eqnarray}
Note that  the matrix $\boldsymbol{C}$ can be computed by a simple Gram--Schmidt process once one derives the explicit expression of the matrix ${d\phi(\bar{\x}(t_f))}$. Thus, it suffices to compute the matrix $\frac{\partial \x}{\partial \p_0}(\cdot,\bar{\p}_0)$ on $[0,t_f]$ and the matrix $\frac{\partial \p^T}{\partial \p_0}(\cdot,\bar{\p}_0)$ at $t_f$ in order to test {\it Conditions}  \ref{AS:Disconjugacy_bang}, \ref{AS:Transversality}, and \ref{AS:terminal_condition}.


It follows from the classical results about solutions to ODEs that the extremal trajectory $(\x(t,{\p}_0),\p(t,\p_0))$ and its time derivative are continuously differentiable with respect to $\p_0$ on $[0,t_f]$. Thus, taking derivative of Eq.~(\ref{EQ:cannonical}) with respect to $\p_0$ on each segment $(t_i,t_{i+1})$, we obtain
\begin{eqnarray}
\left[\begin{array}{c}
\frac{d}{dt}\frac{\partial \x}{\partial\p_0}(t,\bar{\p}_0) \\
\frac{d}{dt}\frac{\partial \p^T}{\partial\p_0}(t,\bar{\p}_0) 
\end{array}
\right]= 
\left[\begin{array}{cc}
H_{\p\x}(\bar{\x}(t),\bar{\p}(t)) & H_{\p\p}(\bar{\x}(t),\bar{\p}(t))\\
-H_{\x\x}(\bar{\x}(t),\bar{\p}(t)) &- H_{\x\p}(\bar{\x}(t),\bar{\p}(t))
\end{array}\right]
\left[
\begin{array}{c}
\frac{\partial \x}{\partial \p_0}(t,\bar{\p}_0)\\
\frac{\partial \p^T}{\partial \p_0}(t,\bar{\p}_0)
\end{array}
\right].
\label{EQ:Homogeneous_matrix}
\end{eqnarray}
Since the initial point $\x_0$ is fixed, one can  obtain the initial conditions  as
\begin{eqnarray}
\frac{\partial \x}{\partial \p_0}(0,\bar{\p}_0) = \boldsymbol{0}_n, \ \frac{\partial \p^T}{\partial \p_0}(0,\bar{\p}_0) = I_n,
\label{EQ:initial_condition}
\end{eqnarray}
where $\boldsymbol{0}_n$ and $I_n$ denote the zero and identity matrix of $\mathbb{R}^{n\times n}$.
Note that the two matrices $\frac{\partial \x}{\partial \p_0}(\cdot,\bar{\p}_0)$ and $\frac{\partial \boldsymbol{p}^T}{\partial \p_0}(\cdot,\bar{\p}_0)$ are discontinuous at the each switching time $t_i$. Comparing with the development in Refs. \cite{Schattler:12,Noble:02,Caillau:15}, the updating formulas for the two matrices  $\frac{\partial \x}{\partial \p_0}(\cdot,\bar{\p}_0)$ and $\frac{\partial \boldsymbol{p}^T}{\partial \p_0}(\cdot,\bar{\p}_0)$ at each switching time $t_i$ can be written as 
\begin{eqnarray}
\frac{\partial \x}{\partial \p_0}(t_i+,\bar{\p}_0) &=& \frac{\partial \x}{\partial \p_0}(t_i-,\bar{\p}_0)  - \Delta \rho_i \f_1(\x(t_i),\boldsymbol{\omega}(t_i)){d t_i(\bar{\p}_0)},\label{EQ:update_formula_x}\\
\frac{\partial \p^T}{\partial \p_0}(t_i+,\bar{\p}_0) &=& \frac{\partial \p^T}{\partial \p_0}(t_i-,\bar{\p}_0) + \Delta \rho_i \frac{\partial \f_1}{\partial \x}(\x(t_i),\boldsymbol{\omega}(t_i)\p^T(t_i){d t_i(\bar{\p}_0)},\label{EQ:update_formula_p}
\end{eqnarray}
where $\Delta \rho_i = \rho(t_i+) - \rho(t_i -)$.  Up to now, except for ${d t_i(\bar{\p}_0)}$, all necessary quantities can be computed. Note that for every $\p_0\in\mathcal{P}$ there holds
\begin{eqnarray}
H_1(\x(t_i(\p_0),\p_0),\p(t_i(\p_0),\p_0)) = 0.
\label{EQ:H_1(t_i)}
\end{eqnarray}
Taking into account $\dot{H}_1(\x(t),\p(t)) = H_{01}(\x(t),\p(t))$, see Eq.~(\ref{EQ:H01}), and differentiating Eq.~(\ref{EQ:H_1(t_i)}) with respect to $\p_0$ yields
\begin{eqnarray}
0 &=& {H}_{01}(\x(t_i,\p_0),\p(t_i,\p_0))dt_i({\p_0}) + \p(t_i,\p_0)\frac{\partial \f_1}{\partial \x}(\x(t_i,\p_0),\boldsymbol{\omega}(t_i,\p_0))\frac{\partial \x(t_i,\p_0)}{\partial \p_0}\nonumber\\
& +& \f_1^T(\x(t_i,\p_0),\boldsymbol{\omega}(t_i,\p_0))\frac{\partial \p^T(t_i,\p_0)}{\partial \p_0}.\nonumber
\end{eqnarray}
According to Assumption \ref{AS:Regular_Switching}, there holds $H_{01}(\bar{\x}(t_i),\bar{\p}(t_i)) \neq 0$ for $i=1,2,\cdots,k$.  Thus, we obtain
\begin{eqnarray}
{d t_i}{( \bar{\p}_0)}  &=& -\Big[\p(t_i,\bar{\p}_0)\frac{\partial \f_1}{\partial \x}(\x(t_i,\p_0),\boldsymbol{\omega}(t_i,\p_0))\frac{\partial \x(t_i,\bar{\p}_0)}{\partial \p_0} \nonumber\\
&+& \f_1^T(\x(t_i,\bar{\p}_0),\boldsymbol{\omega}(t_i,\bar{\p}_0))\frac{\partial \p^T(t_i,\bar{\p}_0)}{\partial \p_0}\Big]/{H}_{01}(\bar{\x}(t_i),\bar{\p}(t_i)).\nonumber
\end{eqnarray} 
Therefore, in order to compute the two matrices $\frac{\partial \x}{\partial \p_0}(\cdot,\bar{\p}_0)$ and $\frac{\partial \boldsymbol{p}^T}{\partial \p_0}(\cdot,\bar{\p}_0)$ on $[0,t_f]$, it is sufficient to choose  the initial condition in Eq.~(\ref{EQ:initial_condition}), then to numerically integrate  Eq.~(\ref{EQ:Homogeneous_matrix}) and to use the updating formulas in Eq.~(\ref{EQ:update_formula_x}) and Eq.~(\ref{EQ:update_formula_p}) once a switching point is encountered. 

According to the approach of Chen {\it et al.} in Ref.~\cite{Caillau:15}, 
given every bang-bang extremal $\Pi(\gamma(\cdot,\bar{\p}_0))$ on $[0,t_f]$, $\delta(\cdot)$ is a constant on every zero-thrust arc. Hence, to test focal points (or conjugate points for $l=n$), it suffices to test the zero of $\delta(\cdot)$ on each maximum-thrust arc and to test the non-positivity of $\delta(t_i-)\delta(t_i+) $ at each switching time $t_i$.

\section{Orbital Transfer Computation}\label{SE:Numerical}

In this numerical section, we consider the three-body problem of the Earth, the Moon, and an artificial spacecraft. Since the orbits of the Earth and the Moon around their common centre of mass are nearly circular, i.e., the eccentricity is around $5.49\times 10^{-2}$, and the mass of an artificial spacecraft is negligible compared with that of the Earth and the Moon,   the Earth-Moon-Spacecraft (EMS) system can be approximately considered as a CRTBP, see Ref.~\cite{Szebehely:67}. Then, we have the below physical parameters corresponding to the EMS, $\mu = 1.2153\times 10^{-2}$, $d_* = 384,400.00$ km, $t_* = 3.7521\times 10^{5}$ seconds, and $m_* = 6.045\times 10^{24}$ kg. The initial mass of the spacecraft is specified as $500$ kg, the maximum thrust of the engine equipped on the spacecraft is taken as $1.0$ N, i.e., $$\tau_{max} =1.0 \frac{ t_*^{2}}{m_*d_*},$$ such that the initial maximum acceleration is $2.0\times 10^{-3}$ m$^2$/s. The spacecraft initially moves on a circular Earth geosynchronous orbit lying on the $XY$-plane such that the radius of the initial orbit is $r_g = 42,165.00$ km. When the spacecraft moves to the point on $X$-axis between the Earth and the Moon, i.e., $\parallel \r(0) \parallel = r_g/d_* - \mu$, we start to control the spacecraft to fly to a circular orbit around the Moon with radius $r_m = 13,069.60$ km such that the $L^1$-norm of control is minimized at the fixed final time $t_f = 38.46$ days. Accordingly, the initial state $\x_0 = (\r_0,\v_0,m_0)$ is given as
$$\r_0 = (r_g/d_*-\mu,0,0)^T,\ \v_0=(0,v_g,0)^T,\ \text{and}\ m_0 = 500/m_*,$$
where $v_g$ is the non-dimensional velocity of the spacecraft on the initial orbit, and the explicit expression of the function $\phi$ in Eq.~(\ref{EQ:final_manifold}) can be written as
\begin{eqnarray}
\phi(\x_f) = \left[
\begin{array}{c}
  \frac{1}{2}\parallel \r(t_f) - [1-\mu,0,0]^T\parallel^2 - \frac{1}{2}(r_m/d_*)^2 \\
 \frac{1}{2}\parallel \v(t_f) \parallel^2 - \frac{1}{2}v_m^2 \\
\v^T(t_f)\cdot(\r(t_f) - [1-\mu,0,0]^T) \\
\r^T(t_f)\cdot \1_{Z} \\
\v^T(t_f)\cdot \1_Z
\end{array}
\right],
\label{EQ:function_phi}
\end{eqnarray}
where $\1_z=[0,0,1]^T$ denotes the unit vector of the $Z$-axis of the rotating frame $OXYZ$ and $v_m$ is the  non-dimensional velocity of the spacecraft on the circular orbit around the Moon with radius $r_m$.


We consider the constant mass model in which $\beta = 0$ since this constant mass model can capture the main features of the original problem, see Refs.~\cite{Caillau:15,Caillau:12,Caillau:12time}. In this case, the mass $m$ is a constant parameter instead of a state in the system $\Sigma$, it follows that $\x = (\r,\v)$ and $\p = (\p_r,\p_v)$. Firstly, we compute the extremal $(\bar{\x}(\cdot),\bar{\p}(\cdot))$ on $[0,t_f]$. It suffices to solve a shooting function corresponding to a two-point boundary value problem~\cite{Pan:13}. A simple shooting method is not stable to solve this problem because one usually does not know a priori the structure of the optimal control, and the numerical computations of the shooting function and its differential may be intricate since the shooting function  is not continuous differentiable. We use a regularization procedure \cite{Caillau:12} by smoothing the control corner to get an energy-optimal trajectory firstly, then use a homotopy method to solve the real trajectory with a bang-bang control. Note that both the initial point $\x_0$ and the final constraint submanifold $\mathcal{M}$ lie on the $XY$-plane, it follows that the whole trajectory lies on the  $XY$-plane as well. Fig. \ref{Fig:Transferring_Orbit3_1} illustrates the non-dimensional profile of the position vector $ \r $ along the computed extremal trajectory.
\begin{figure}[!ht]
 \centering\includegraphics[width=\textwidth, angle=0]{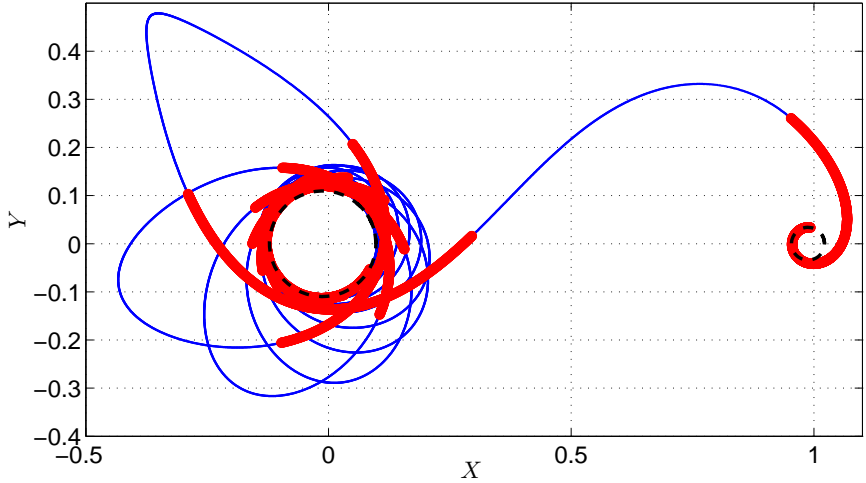}
 \caption[]{The non-dimensional profile of the position vector $ \r $ of the $L^1$-minization trajectory in the rotating frame $OXYZ$ of the EMS system. The thick curves are the maximum-thrust arcs and the thin curves are the zero-thrust arcs. The bigger dashed circle and the smaller one are the initial and final circular orbits around the Earth and the Moon, respectively.}
 \label{Fig:Transferring_Orbit3_1}
\end{figure}
The profiles of $\rho$, $\parallel \p_v \parallel$, and $H_1$ with respect to non-dimensional time are shown in Fig. \ref{Fig:Transferring_Orbit4_1}, from which we can see that the number of maximum-thrust arcs is 15 with 29 switching points and that the ragularity condition in Assumption \ref{AS:Regular_Switching} at every switching point is satisfied.
\begin{figure}[!ht]
 \centering\includegraphics[ width=\textwidth, angle=0]{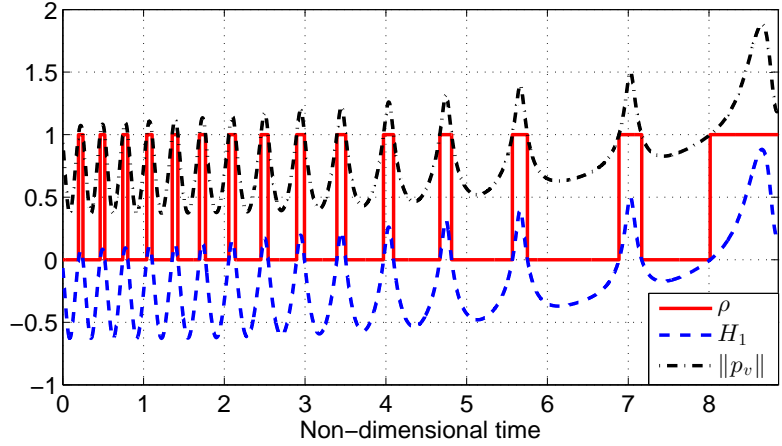}
 \caption[]{The profiles of $\rho$, $\parallel \p_v \parallel $, and $H_1$ with respect to non-dimensional time along the $L^1$-minimization trajectory.}
 \label{Fig:Transferring_Orbit4_1}
\end{figure}
Since the extremal trajectory is computed based on necessary conditions, one has to check sufficient optimality conditions  to make sure that it is at least locally optimal. According to what has been developed in Section \ref{SE:Sufficient}, it suffices to check the satisfaction of {\it Conditions} \ref{AS:Disconjugacy_bang}, \ref{AS:Transversality}, and \ref{AS:terminal_condition}. 
Using Eqs.~(\ref{EQ:Homogeneous_matrix}--\ref{EQ:update_formula_p}), one can compute $\delta(\cdot)$ on $[0,t_f]$. In order to have a clear view, the profile of $\delta(\cdot)$ on $[0,t_f]$ is rescaled 
\begin{figure}[!ht]
 \centering\includegraphics[width=\textwidth, angle=0]{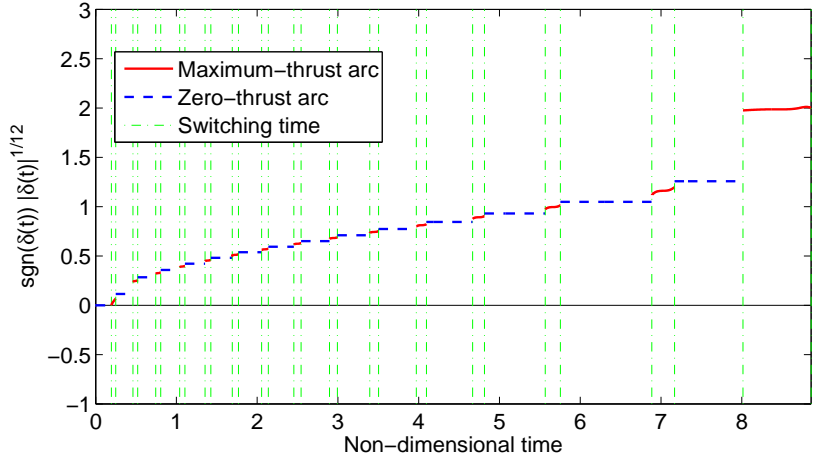}
 \caption[]{The profile of $sgn(\delta(t))|\delta(t)|^{1/12}$ with respect to non-dimensional time along the $L^1$-minimization extremal  in EMS.}
 \label{Fig:det}
\end{figure}
by $\text{sgn}(\delta(\cdot))*|\delta(\cdot)|^{1/12}$, which can capture the sign property of $\delta(\cdot)$ on $[0,t_f]$, as is illustrated in Fig. \ref{Fig:det}. We can see that there exist  no sign changes at each switching point  and no zeros on each smooth bang arc.  Thus,  Conditions \ref{AS:Disconjugacy_bang} and \ref{AS:Transversality} are satisfied along the computed extremal. To check Condition \ref{AS:terminal_condition}, differentiating $\phi(\cdot)$ in Eq.~(\ref{EQ:function_phi}) yields 
\begin{eqnarray}
{d \phi(\bar{\x}(t_f)) }=
\left[\begin{array}{ccccc}
\r(t_f) - [1-\mu,0,0]^T & \boldsymbol{0}_{3\times 1} & \v(t_f)& \1_Z& \boldsymbol{0}_{3\times 1}\\
 \boldsymbol{0}_{3\times 1} &\v(t_f) &\r(t_f) - [1-\mu,0,0]^T & \boldsymbol{0}_{3\times 1}&\1_Z
 \end{array}
\right]^T,
\label{EQ:dphi_numerical}
\end{eqnarray}
and
\begin{eqnarray}
d^2{\phi_1(\bar{\x}(t_f))} &=& \left(\begin{array}{cc} I_3 &\boldsymbol{0}_3\\
\boldsymbol{0}_3 & \boldsymbol{0}_3
\end{array}\right),\ \ d^2{\phi_2(\bar{\x}(t_f))}= \left(\begin{array}{cc} \boldsymbol{0}_3 &\boldsymbol{0}_3\\
\boldsymbol{0}_3 & I_3
\end{array}\right),\nonumber\\
d^2{\phi_3(\bar{\x}(t_f))} &=& \left(\begin{array}{cc} \boldsymbol{0}_3 & I_3\\
I_3 & \boldsymbol{0}_3
\end{array}\right),\ \ d^2{\phi_4(\bar{\x}(t_f))} = d^2{\phi_5(\bar{\x}(t_f))}= \boldsymbol{0}_6,\nonumber
\end{eqnarray}
where $\phi_i(\cdot):\mathcal{X}\rightarrow \mathbb{R},\ \x\mapsto \phi_i(\x)$ for $i=1,2,\cdots,l$ are the elements of the vector-valued function $\phi(\x)$.
Then, substituting the values of $\bar{\x}(t_f)$ and $\bar{\p}(t_f)$ into Eq.~(\ref{EQ:numerical_nu}), the vector $\bar{\boldsymbol{\nu}}$ can be computed. Up to now, except the matrix $\boldsymbol{C}$, all the quantities in Eq.~(\ref{EQ:positive_definite}) are obtained. Actually, one can use a Gram-Schmidt process to compute the matrix $\boldsymbol{C}$ associated with the matrix in Eq.~(\ref{EQ:dphi_numerical}). Then, substituting numerical values into Eq.~(\ref{EQ:positive_definite}), we obtain 
$$\boldsymbol{C}^T\left\{\frac{\partial \p^T(t_f,\bar{\p}_0)}{\partial \p_0} \left[ \frac{\partial \x(t_f,\bar{\p}_0)}{\partial \p_0}\right]^{-1} - \bar{\boldsymbol{\nu}} d^2\phi(\bar{\x}(t_f))\right\}\boldsymbol{C} \approx 0.5292 \succ 0.$$
Thus, Condition \ref{AS:terminal_condition} is satisfied. Note that the dimension of the submanifold $\mathcal{M}$ is one, it follows that the smooth curve $\y(\cdot)\in\mathcal{M}\cap\mathcal{N}$ on $[-\varepsilon,\varepsilon]$ for every $\varepsilon > 0$ is a one-dimensional curve restricted on the final circular orbit around the Moon. 
Fig. \ref{Fig:J_xi}
\begin{figure}[!ht]
 \centering\includegraphics[width=\textwidth, angle=0]{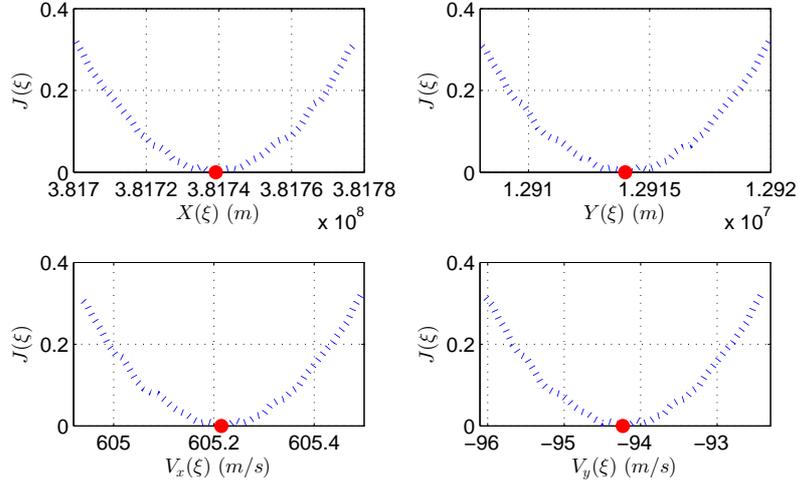}
 \caption[]{Let $X(\xi)$ and $Y(\xi)$ be the projection of the position vector $\r(\xi)$ on $X$- and $Y$-axis of the rotating frame $OXYZ$, respectively, and let $V_x(\xi)$ and $V_y(\xi)$ be the projection of the velocity vector $\v(\xi)$ on $X$- and $Y$-axis of the rotating frame $OXYZ$, respectively. The figure plots the profiles $J(\xi)$ with respect to $X(\xi)$, $Y(\xi)$, $V_x(\xi)$, and $V_y(\xi)$. The dots on each plot denote $(J(0),\y(0))$.}
 \label{Fig:J_xi}
\end{figure}
 illustrates the profile of $J(\cdot)$ with respect to $\y(\cdot)\in\mathcal{M}\cap\mathcal{N}$ in a small neighbourhood of $\bar{\x}(t_f)$. we can clearly see that $J(\cdot)>J(0)$ on $[-\varepsilon,\varepsilon]\backslash\{0\}$.  Up to now,  all the conditions in Theorem \ref{TH:optimality} are satisfied. So, the computed $L^1$-minimization trajectory realizes  a strict strong-local optimality in $C^{0}$-topology.

\section{Conclusions}

In this paper, the PMP is first employed to formulate the Hamiltonian system of the $L^1$-minimization problem for the translational motion of a spacecraft in the CRTBP, showing that the optimal control functions can exhibit bang-bang and singular behaviors. Moreover, the singular extremals are of at least order two, revealing the existence of Fuller or chattering phenomena. To establish the sufficient optimality conditions, a parameterized family of extremals is constructed. As a result of analyzing the projection behavior of this family, we obtain that conjugate points may occur not only on maximum-thrust arcs between switching times but also at switching times. Directly applying the theory of field of extremals, we obtain that the disconjugacy conditions (cf. Conditions \ref{AS:Disconjugacy_bang} and \ref{AS:Transversality})  are sufficient to guarantee an extremal to be locally optimal if the endpoints are fixed. For the case that the dimension of the final constraint submanifold is not zero, we establish a further second-order condition (cf. Condition \ref{AS:terminal_condition}), which is a necessary and sufficient one for the strict strong-local optimality of a bang-bang extremal if disconjugacy conditions are satisfied. In addition, the numerical implementation for these three sufficient optimality conditions is derived. Finally, an example of transferring a spacecraft from a circular orbit around the Earth to an orbit around the Moon is computed and the second-order sufficient optimality conditions developed in this paper are tested to show that the computed extremal realizes a strict strong-local optimum.  The sufficient optimality conditions for open-time problems will be considered in future work.


\begin{thebibliography}{}
%
%

\bibitem{Brusch:70} Brusch, R. G., and Vincent, T. L.,  { Numerical Implementation of a Second-Order Variational Endpoint Condition}, AIAA Journal, 8(12), 2230--2235 (1970)

\bibitem{Wood:74} Wood, L. J., Second-Order Optimality Conditions for the Bolza Problem with Both Endpoints Variable, Journal of Aircraft, Vol. 11(4), 212--221 (1974)


\bibitem{Pontryagin} Pontryagin, L. S., Boltyanski, V. G., Gamkrelidze R. V., and Mishchenko E. F.,  The Mathematical Theory of Optimal Processes (Russian), English translation: Interscience (1962)

\bibitem{Caillau:12} Caillau, J.-B., Daoud, B., and Gergaud, J.,  Minimum Fuel Control of the Planar Circular Restricted Three-Body Problem,  Celestial Mechanics and Dynamical Astronomy, 114, 137--150 (2012)

\bibitem{Caillau:12time} Caillau, J.-B., and Daoud, B.,  Minimum Time Control of the Restricted Three-Body Problem, SIAM Journal of Control and Optimization, 50(6), 3178--3202 (2012)

\bibitem{Zhang:15} Zhang, C., Topputo, F., Bernelli-Zazzera, F., and Zhao, Y.,  Low-Thrust Minimum-Fuel Optimization in the Circular Restricted Three-Body Problem,  Journal of Guidance, Control, and Dynamics, 38(8), 1501--1510 (2015)

\bibitem{Mingotti:09} Mingotti, G., Topputo, F., and Bernelli-Zazzera, F.,  Low-Energy, Low-Thrust Transfers to the Moon, Celest. Mech. Dyn. Astron. 105(1--3), 61--74 (2009)


\bibitem{Ross:07} Ross, S. D., and Scheeres, D. J.,  Multiple Gravity Assists, Capture, and Escape in the Restricted Three-Body Problem, SIAM Journal of Applied Dynamics and Systems, 6(3) 576--596 (2007)

\bibitem{Ghezzi:15} Ghezzi, R., Caponigro, M., Piccoli, B., and Tr\'elat, E.,  Regularization of chattering phenomena via bounded variation controls,  NETCO 2014, Tours, France. $<hal-01024604>$ (2014)

\bibitem{Park:13} Park, C.,  Necessary Conditions for the Optimality of Singular Arcs of Spacecraft Trajectories subject to Multiple Gravitational Bodies, Advances in Space Research, 51(11), 2125--2135 (2013)



\bibitem{Ozimek:10} Ozimek, M. T., and Howell, K. C.,  Low-Thrust Transfers in the Earth-Moon System, Including Applications to Libration Point Orbits,  Journal of Guidance, Control, and Dynamics, 33(2), 533--549 (2010)


\bibitem{Zelikin:94} Zelikin, M. I., and Borisov, V. F.,   Theory of Chattering Control with Applications to Astronautics, Robotics, Economics, and Engineering,  Birkhauser (1994)

\bibitem{Zelikin:03} Zelikin, M. I., and Borisov, V. F.,  Optimal Chattering Feedback Control,  Journal of Mathematical Sciences, 114(3),1227--1344 (2003)

\bibitem{Marchal:73} Marchal, C., Chattering Arcs and Chattering Controls, Journal of Optimization Theory and Applications, 11(5), 441--468 (1973)

\bibitem{Robbins:65} Robbins, H. M., Optimality of intermediate-thrust arcs of rocket trajectories, AIAA Journal, 3(6), 1094--1098 (1965)


\bibitem{Noble:02} Noble, J. and Sch\"attler, H.,  Sufficient Conditions for Relative Minima of Broken Extremals in Optimal Control Theory, Journal of Mathematical Analysis and Applications,  269, 98-128 (2002)

\bibitem{Schattler:12} Sch\"{a}ttler, H. and Ledzewicz, U.,  Geometric Optimal Control: Theory, Methods, and Examples, Springer, (2012)

\bibitem{Agrachev:04} Agrachev, A. A. and Sachkov, Y. L.,  Control Theory from the Geometric Viewpoint, Encyclopedia of Mathematical Sciences, 87, Control Theory and Optimization, II. Springer-Verlag, Berlin (2004)

\bibitem{Kupka:87} Kupka, I., Geometric Theory of Extremals in Optimal Control Problems I; The Fold and Maxwell Case, Trans. Amer. Math. Soc. 299(1), 225-243 (1987)

\bibitem{Sarychev:82} Sarychev, A. V.,  The Index of Second Variation of a Control System, Mat. Sb. 41, 338-401 (1982)



\bibitem{Sussmann:85} Sussmann, H. J.,  Envelopes, Conjugate Points and Optimal Bang-Bang Extremals,  in Proc. 1985 Paris Conf. on Nonlinear Systems, Fliess, M., and Hazewinkel, M., eds., Reidel Publishers, Dordrecht, the Netherlands, (1987)


\bibitem{Caillau:15} Chen, Z., Caillau, J.-B.,  and Chitour, Y.,  $L^1$-Minimization for Mechanical Systems, arXiv:1506.00569 (2015).


\bibitem{Lawden:63} Lawden, D. F.,  Optimal Trajectories for Space Navigation, Butterworth, London (1963)



\bibitem{Pan:13} Pan, B., Chen, Z., Lu, P., and Gao, B.,  Reduced Transversality Conditions for Optimal Space Trajectories, Journal of Guidance, Control, and Dynamics, 36(5), 1289-1300 (2013)







\bibitem{Bryson:69} Bryson, A. E., Jr. and Ho, Y. C.,  Applied Optimal Control, Blaisdell, Waltham, Mass., 177-211 (1969)

\bibitem{Mermau:76} Mermau, P. M., and Powers, W. F.,   Conjugate Point Properties for Linear Quadratic Problems, Journal of Mathematical Analysis and Applications, 55, 418-433 (1976)

\bibitem{Breakwell:65} Breakwell, J. V., and Ho, Y. C.,  On the Conjugate Point Condition for the Control Problem, International Journal of Engineering Science,  2, 565-579 (1965)




\bibitem{Kelley:66} Kelley, H. J., Kopp, R. E., and Moyer, A. G.,  Singular Extremals, Optimization Theory and Applications (G. Leitmann, ed.), Chapter 3, Academic Press  (1966)







\bibitem{Bonnard:07} Bonnard, B., Caillau, J.-B., and Tr\'elat, E.,  Second-Order Optimality Conditions in the Smooth Case and Applications in Optimal Control, ESAIM Control Optimization and Calculus of Variation, 13(2), 207-236 (2007)

\bibitem{Poggiolini:04} Poggiolini, L. and Stefani, G.,  State-Local Optimality of a Bang-Bang Trajectory: a Hamiltonian Approach, Systems $\&$ Control Letters,53, 269-279 (2004)

\bibitem{Agrachev:02} Agrachev, A. A., Stefani, G., and Zezza, P., Strong Optimality for a Bang-Bang Trajectory, SIAM Journal of Control and Optimization,41(4), 1991-1041 (2002)







\bibitem{Szebehely:67} Szebehely, V.,  Theory of Orbits: The Restricted Problem of Three Bodies, Academic Press, Massachusetts (1967)





\bibitem{Gergaud:06} Gergaud, J., and Haberkorn, T.,  Homotopy Method for Minimum Consumption Orbital Transer Problem, ESAIM: Control, Optimization and Calculus of Variations, 12, 294-310 (2006)













%















\end{thebibliography}
\end{document}